\documentclass{amsart}
\usepackage[utf8]{inputenc}

\usepackage{fullpage}
\usepackage[utf8]{inputenc}
\usepackage[T1]{fontenc}
\usepackage{graphicx,subfigure}
\usepackage{amsmath,amsfonts,amssymb,mathtools}
\usepackage{stmaryrd}
\usepackage{mathrsfs,dsfont}
\usepackage{amsthm}
\usepackage{mathabx}
\usepackage{tabularx}
\usepackage{float}
\usepackage[usenames]{color}
\allowdisplaybreaks
\usepackage{hyperref}

\newtheorem{theo}{Theorem}[section]
\newtheorem{lemma}[theo]{Lemma}
\newtheorem{ass}{Assumption}[section]

\newtheorem{propo}{Proposition}[section]

\newtheorem{cor}{Corollary}[section]

\newcommand{\N}{\mathbb{N}}
\newcommand{\R}{\mathbb{R}}
\newcommand{\E}{\mathbb{E}}

\newcommand{\IL}{\Lambda}
\newcommand{\HH}{\mathcal{H}}
\newcommand{\Mg}{M(\gamma)}
\newcommand{\rg}{\rho(\gamma)}
\newcommand{\IW}{\mathcal{W}}
\newcommand{\Lf}{{\rm L}_f}
\newcommand{\IZ}{\mathcal{Z}}
\newcommand{\IP}{\mathcal{P}}
\newcommand{\NN}{\overline{N}}

\title{Numerical approximation of the invariant distribution for a class of stochastic damped wave equations}
\author{Ziyi Lei$^{1}$}
\address{$^{1}$ Central South University, HNP-LAMA, School of Mathematics and Statistics, Changsha, China}
\email{csu\_ziyilei@csu.edu.cn, sqgan@csu.edu.cn}

\author{Charles-Edouard Br\'ehier$^{2}$}
\address{$^{2}$ Universit\'e de Pau et des Pays de l'Adour, E2S UPPA, CNRS, LMAP, Pau, France}
\email{charles-edouard.brehier@univ-pau.fr}

\author{Siqing Gan$^{1}$}

\date{}

\begin{document}

\begin{abstract}
We study a class of stochastic semilinear damped wave equations driven by additive Wiener noise. Owing to the damping term, under appropriate conditions on the nonlinearity, the solution admits a unique invariant distribution. We apply semi-discrete and fully-discrete methods in order to approximate this invariant distribution, using a spectral Galerkin method and an exponential Euler integrator for spatial and temporal discretization respectively. We prove that the considered numerical schemes also admit unique invariant distributions, and we prove error estimates between the approximate and exact invariant distributions, with identification of the orders of convergence. To the best of our knowledge this is the first result in the literature concerning numerical approximation of invariant distributions for stochastic damped wave equations.
\end{abstract}


\maketitle

\section{Introduction}

In the last decades, stochastic partial differential equations (SPDEs) have become the subject of intensive research, with various possible perspectives, from modelling and applications, to theoretical analysis using advanced stochastic and PDE analysis techniques. Due to the need of simulating efficiently the SPDE models, there has been a huge interest in the design and analysis of numerical methods, see for instance the monograph~\cite{LPS}. Like for all deterministic and stochastic, finite and infinite dimensional, systems, understanding the long-time behavior of the solutions and of the numerical methods is a crucial and challenging problem. In this manuscript, we provide a contribution in this direction for a class of ergodic stochastic damped semilinear wave equations. Precisely, we apply semi-discrete and fully-discrete schemes and we prove that the approximations preserve ergodicity of the exact system, and we obtain upper bounds with rates for the error between the approximate and exact invariant distributions.

In this work, we study stochastic damped semilinear wave equations driven by additive Wiener noise, which can formally be written as
\begin{equation}\label{eq:spde_intro1}
\left\lbrace
\begin{aligned}
&\partial_{t}^2u(t,x)=-2\gamma\partial_tu(t,x)+\Delta u(t,x)+f(u(t,x))+\dot{W}^Q(t,x),\\
&u(0,x)=u_0(x),~\partial_tu(0,x)=v_0(x),~\quad x\in \mathcal{D},\\
&u(t,x)=0,~\quad x\in \partial \mathcal{D},
\end{aligned}
\right.
\end{equation}
where $\mathcal{D}$ is a $d$-dimensional domain, $\Delta=\partial_{x_1}^2+\ldots+\partial_{x_d}^2$ is the Laplace operator, and homogeneous Dirichlet boundary conditions are imposed. The positive real number $\gamma\in(0,\infty)$ is the damping parameter. In addition, $u_0:\mathcal{D}\to \R$ and $v_0:\mathcal{D}\to \R$ denote the initial values of $u$ and $\partial_tu$. Finally, $f:\R\to\R$ is a globally Lipschitz nonlinearity, and $\dot{W}^Q$ denotes space-time noise, which is white in time and has covariance operator $Q$ in space. We refer to Section~\ref{sec:setting} below for precise formulations and assumptions. In this manuscript, the stochastic damped wave equation~\eqref{eq:spde_intro1} is interpreted as a stochastic evolution system with unknowns $u(t)\in H$ and $v(t)\in H^{-1}$
\begin{equation}\label{eq:spde_intro2}
\left\lbrace
     \begin{aligned}
&du(t)=v(t)dt,\\
&dv(t)=(-\IL u(t)-2\gamma v(t))\,dt+f(u(t))\,dt+\,dW^Q(t),\\
&u(0)=u_0,\  \  v(0)=v_0, 
     \end{aligned}
\right.
\end{equation}
or as a stochastic evolution equation with unknown $X(t)=(u(t),v(t))$ with values in $\HH=H\times H^{-1}$, written as
\begin{equation}\label{eq:see_intro}
\left\lbrace
     \begin{aligned}
&dX(t)=A_\gamma X(t)dt+F(X(t))dt+d\IW^Q(t),\\
&X(0)=x_0.
     \end{aligned}
\right.
\end{equation}
See Section~\ref{sec:setting} for details on the notation. We refer to the monograph~\cite{DPZ} for the theory of stochastic evolution equations, and to~\cite{MR1500166} for a chapter devoted to the analysis of the stochastic wave equation. In the setting introduced below, the stochastic evolution system~\eqref{eq:see_intro} admits a unique global mild solution, for any initial value.

Assuming that the Lipschitz constant of the nonlinearity $f$ is sufficiently small (see condition~\eqref{eq:ergo} from Assumption~\ref{ass:ergo} below), the process $\bigl(X(t)\bigr)_{t\ge 0}$ admits a unique invariant distribution denoted by $\mu_{\infty}$, see Proposition~\ref{propo:ergo-exact} below, and the monograph~\cite{DPZergo} for general results on ergodicity for stochastic evolution equations. The objective of this work is to prove that $\mu_\infty$ can be approximated using fully-discrete numerical methods, and to provide rates of convergence with respect to the spatial and temporal discretization parameters. Precisely, the spatial discretization is performed using a spectral Galerkin method with parameter denoted by $N\in\N$, and the temporal discretization is performed using an exponential Euler integrator with time-step size denoted by $\tau\in(0,1)$. We refer to Section~\ref{sec:numerics} for precise notation. The fully-discrete scheme (see Section~\ref{sec:fully}) is written as
\begin{equation}\label{eq:expoEuler_intro}
X^N_{m+1}=e^{\tau A_{\gamma}}\Bigl(X^N_m+\tau F_N(X^N_m)+\IP_N\Delta \IW^Q_m\Bigr),
\end{equation}
where $\bigl(e^{t A_\gamma}\bigr)_{t\ge 0}$ is the semigroup associated with the linear damped wave equation with no forcing. We also consider a semi-discrete scheme (see Section~\ref{sec:galerkin})
\begin{equation}\label{eq:seegalerkin_intro}
dX^N(t)=A_\gamma X^N(t)dt+F_N(X^N(t))dt+\IP_Nd\IW^Q(t)
\end{equation}
where only spatial discretization is performed.

We prove that the fully-discrete scheme~\eqref{eq:expoEuler_intro} and the semi-discrete scheme~\eqref{eq:seegalerkin_intro} admit unique invariant distributions denoted by $\mu_\infty^{N,\tau}$ and $\mu_\infty^{N}$ respectively, for any values of the spatial and temporal discretization parameters $N$ and $\tau$, see Propositions~\ref{propo:ergo-eem} and~\ref{propo:ergo-galerkin} for precise statements. The main results of this work are then stated in Section~\ref{sec:main}: for the fully-discrete scheme~\eqref{eq:expoEuler_intro} one has weak error estimates
\begin{equation}\label{eq:main_intro}
\big|\int \varphi d\mu_\infty-\int \varphi d\mu_\infty^{N,\tau}\big|\le C_{\beta}(\varphi)\Bigl( \lambda_N^{-\beta\theta}+\tau^{\min(2\theta\beta,1)}\Bigr),
\end{equation}
where $\varphi:\HH\to \R$ is an arbitrary mapping of class $\mathcal{C}^2$ with bounded first and second order derivatives, $\bigl(\lambda_n\bigr)_{n\in\N}$ are the eigenvalues of $\IL=-\Delta$ (written in non-decreasing order), and $\beta\in(0,1]$ is a parameter which describes the regularity of the noise, see~\eqref{eq:Qbeta} from Assumption~\ref{ass:Q}. Let us mention that one has $\beta\in(0,\frac12)$ for space-time white noise ($Q=I$, $d=1$) and $\beta=1$ for trace-class noise. Finally, $\theta=1$ if the covariance operator $Q$ and the linear operator $\IL=-\Delta$ commute, and $\theta=1/2$ otherwise. We refer to Theorem~\ref{theo:galerkin} for a precise statement concerning the semi-discrete scheme~\eqref{eq:seegalerkin_intro} and to Theorem~\ref{theo:eem} (and to Corollary~\ref{cor:full}) for a precise statement concerning the fully-discrete scheme~\eqref{eq:expoEuler_intro}.

The rates of convergence appearing in~\eqref{eq:main_intro} are consistent with the weak rates of convergence obtained for the stochastic wave equation for instance in~\cite{MR3353942,MR4356894,MR3884775} (without damping, meaning $\gamma=0$). To the best of our knowledge, the weak error estimates~\eqref{eq:main_intro} have not been obtained previously in the literature. One of the main auxiliary tools used to prove the weak error estimates is Proposition~\ref{propo:Phi} concerning properties of solutions of Kolmogorov equations associated with the semi-discrete scheme~\eqref{eq:seegalerkin_intro}, with estimates which are independent of the integer $N$. We refer to the monograph~\cite{Cerrai} for results and methods on the analysis of finite and infinite dimensional Kolmogorov equations.

Let us briefly review the literature related to our work. The list of references below is not exhaustive. Strong convergence rates for spatial and temporal discretization schemes applied to stochastic wave equations have been obtained for instance in the articles~\cite{MR3353942,MR3276429,MR3484400,MR2646102,MR3033008} for semi-discrete and fully-discrete methods using the interpretation of SPDEs of stochastic evolution equations from~\cite{DPZ}, which is considered in this manuscript. The works mentioned above employ exponential or trigonometric integrator for the temporal discretization, see also~\cite{JentzenKloeden,MR2728063} for a presentation of exponential integrators in the context of SPDEs. The articles~\cite{MR2224753,MR3463447} provide strong convergence analysis using the point of view of random field solutions for SPDEs, developed by~\cite{MR876085}, see also~\cite{MR1500166} for a focus on the stochastic wave equation. We also refer to the recent work~\cite{feng2022higher} where higher order methods are constructed. The authors of the preprint~\cite{CCW} consider a class of stochastic wave equations with nonlinear damping and prove strong convergence results. Weak convergence rates for numerical methods applied to stochastic wave equations have been obtained in~\cite{MR2671031}, and more recently in~\cite{MR4356894,MR3884775} and in~\cite{MR3353942}. The framework we consider below is closely related to the latter reference, however we need to use different techniques in the analysis: the author of~\cite{MR3353942} exploits the group property associated with the linear wave equation when there is no damping, whereas to take into account the damping and prove error bounds which are uniform with respect to time we only use semigroup properties.

As already mentioned above, to the best of our knowledge, in this manuscript we obtain the first results concerning the approximation of the invariant distribution for stochastic damped wave equations. However, there have been many results for different types of finite and infinite dimensional stochastic systems. Some of the arguments used in the analysis below are inspired by the articles cited below. For stochastic differential equations, one of the earliest works is~\cite{Talay}. Writing the stochastic damped wave equation~\eqref{eq:see_intro} as the system~\eqref{eq:spde_intro2} suggests an analogy with the works~\cite{Talay:02} and~\cite{MR3434031} devoted to stochastic Hamiltonian systems and Langevin systems respectively. Note that the monograph~\cite{Hong-Wang} exploits a similar point of view for the approximation of invariant distributions of stochastic nonlinear Schr\"odinger equations, see also the articles~\cite{MR3605170,MR3608750}. For further references and techniques for the design and analysis of numerical methods applied to approximate invariant distributions of stochastic differential equations, we refer for instance to~\cite{AbdulleAlmuslimaniVilmart,AbdulleVilmartZygalakis,CrisanDobsonOttobre:21,LambertonPages,LaurentVilmart,LeimkuhlerMatthewsTretyakov,MilsteinTretyakov,Vilmart}. Finally, in the last decade, there have been several articles devoted to the numerical approximation of invariant distributions for parabolic semilinear stochastic partial differential equations: we refer to~\cite{C-E:14,C-E:16_2,C-E:17,C-E:22}, to~\cite{Chen-Gan-Wang:20,Cui-Hong-Sun}, and to~\cite{MR4454935}. See also the recent preprint~\cite{GHM} concerning the analysis of long-term accuracy of numerical methods for some SPDEs. The analysis of the stochastic damped wave equation case requires new and different arguments compared with the parabolic case.

This paper is organized as follows. The setting (notation and assumptions) is introduced in Section~\ref{sec:setting}. Section~\ref{sec:spde} is then devoted to the presentation and analysis of the stochastic damped wave equation~\eqref{eq:see_intro}. The semi-discrete and fully-discrete numerical methods are defined in Sections~\ref{sec:galerkin} and~\ref{sec:fully} respectively, and the main results of this paper are stated in Section~\ref{sec:main}: see Theorem~\ref{theo:galerkin} for the spatial approximation, Theorem~\ref{theo:eem} for the temporal approximation, and Corollary~\ref{cor:full} for the fully-discrete scheme. The convergence analysis is performed in Section~\ref{sec:proofs}: note that an important regularity result for the associated Kolmogorov equation is given by Proposition~\ref{propo:Phi} in Section~\ref{sec:aux}. Sections~\ref{sec:proofgalerkin} and~\ref{sec:prooffully} are devoted to the proofs of Theorems~\ref{theo:galerkin} and~\ref{theo:eem} respectively.

\section{Setting}\label{sec:setting}

\subsection{Notation}

Let $d\in\N$ be an integer 	and let $\mathcal{D}\subset\R^d$ be an open bounded domain with polygonal boundary $\partial\mathcal{D}$, for instance $\mathcal{D}=(0,1)^d$. Define the separable Hilbert space $H=L^2(\mathcal{D})$. The inner product and the norm are denoted by $\langle \cdot,\cdot \rangle_H$ and $\|\cdot\|_H$ respectively. Let us denote by $-\IL$ the Laplace operator endowed with homogeneous Dirichlet boundary conditions. Then $\IL$ is a self-adjoint unbounded linear operator from $H$ to $H$, with domain $D(\IL)=H^2(\mathcal{D})\cap H_0^1(\mathcal{D})$, and there exists a complete orthonormal system $\bigl(e_n\bigr)_{n\in\N}$ of $H$ and a nondecreasing sequence $\bigl(\lambda_n\bigr)_{n\in\N}$ of positive real numbers such that one has
\[
\Lambda e_n=\lambda_n e_n.
\]
Note that the eigenvalue $\lambda_1$ is positive and that one has
\[
\lambda_1=\underset{n\in\N}\min~\lambda_n=\underset{x\in H\setminus\{0\}}\inf~\frac{\langle \IL x,x\rangle_H}{\|x\|_H^2}.
\]

In addition, there exists ${\rm c}_d\in(0,\infty)$ such that $\lambda_n\sim {\rm c}_dn^{2/d}$ when $n\to\infty$.

For any nonnegative real number $\alpha\in[0,+\infty)$, introduce the subspace $H^\alpha$ of $H$ defined by
\[
H^\alpha=\{u\in H;~\sum_{n=1}^{\infty}\lambda_n^{\alpha}\langle u,e_n\rangle_H^2<+\infty\},
\]
and for all $u,u_1,u_2\in H^\alpha$, set
\begin{align*}
\langle u_1,u_2 \rangle_{H^\alpha}&=\sum_{n=1}^{+\infty}\lambda_n^{\alpha}\langle u_1,e_n \rangle_H \langle u_2,e_n \rangle_H,\\
\|u\|_{H^\alpha}^2&=\langle u,u \rangle_{H^\alpha}=\sum_{n=1}^{+\infty}\lambda_n^{\alpha}\langle u,e_n \rangle_H^2.
\end{align*}
The space $H^\alpha$, equipped with the inner product $\langle \cdot,\cdot \rangle_{H^\alpha}$ and the norm $\|\cdot\|_{H^\alpha}$, is a separable Hilbert space.

For any nonnegative real number $\alpha\in[0,+\infty)$, the space $H^{-\alpha}$ is defined as the closure of the space
\[
\{u\in H;~\sum_{n=1}^{+\infty}\lambda_n^{-\alpha}\langle u,e_n\rangle_H^2<\infty\},
\]
with inner product $\langle\cdot,\cdot\rangle_{H^{-\alpha}}$ and norm $\|\cdot\|_{H^{-\alpha}}$ defined as follows: for all $u,u_1,u_2\in H^{-\alpha}$, set
\begin{align*}
\langle u_1,u_2 \rangle_{H^{-\alpha}}&=\sum_{n=1}^{+\infty}\lambda_n^{-\alpha}\langle u_1,e_n \rangle_H \langle u_2,e_n \rangle_H,\\
\|u\|_{H^{-\alpha}}^2&=\langle u,u \rangle_{H^{-\alpha}}=\sum_{n=1}^{+\infty}\lambda_n^{-\alpha}\langle u,e_n \rangle_H^2.
\end{align*}
The space $H^{-\alpha}$ equipped with the inner product $\langle \cdot,\cdot \rangle_{H^{-\alpha}}$ and the norm $\|\cdot\|_{H^{-\alpha}}$ is a separable Hilbert space. It can be identified with the dual space of $H^{\alpha}$. 

For all $\alpha\in\R$, let the linear operator $\IL^\alpha$ be defined as follows: for all $u\in H^\alpha$, set
\[
\IL^{\alpha/2}u=\sum_{n\in\N}\lambda_n^{\alpha/2}\langle u,e_n \rangle_H e_n.
\]

For all $\alpha\in\mathbb{R}$, introduce the space
\[
\HH^\alpha=H^\alpha \times H^{\alpha-1},
\]
and for all $x=(u,v), x_1=(u_1,v_1), x_2=(u_2,v_2)\in\HH^\alpha$, define
\begin{align*}
\langle x_1,x_2 \rangle_{\HH^\alpha}&=\langle (u_1,v_1),(u_2,v_2) \rangle_{\HH^\alpha}=\langle u_1,u_2 \rangle_{H^\alpha}+\langle v_1,v_2 \rangle_{H^{\alpha-1}},\\
\|x\|_{\HH^\alpha}^2&=\|(u,v)\|_{\HH^\alpha}^2=\langle x, x \rangle_{\HH^\alpha}=\|u\|_{H^\alpha}^2+\|v\|_{H^{\alpha-1}}^2.
\end{align*}
For all $\alpha\in\mathbb{R}$, the space $\HH^\alpha$, equipped with the inner product $\langle \cdot,\cdot \rangle_{\HH^\alpha}$ and the norm $\|\cdot\|_{\HH^\alpha}$, is a separable Hilbert space. When $\alpha=0$, the notation $\HH=\HH^0$, $\langle \cdot,\cdot \rangle_{\HH}=\langle \cdot,\cdot \rangle_{\HH^0}$ and $\|\cdot\|_{\HH}=\|\cdot\|_{\HH^0}$ is used in the sequel.

Observe that if $\alpha_1<\alpha_2$, one has $H^{\alpha_2}\subset H^{\alpha_1}$ and $\HH^{\alpha_2}\subset \HH^{\alpha_1}$, with continuous embeddings: for all $u\in H^{\alpha_2}$ and all $x\in\HH^{\alpha_2}$, one has
\begin{equation}\label{eq:comparenorms}
\|u\|_{H^{\alpha_1}}\le \lambda_1^{-\frac{\alpha_2-\alpha_1}{2}}\|u\|_{H^{\alpha_2}}~,\quad \|x\|_{\HH^{\alpha_1}}\le \lambda_1^{-\frac{\alpha_2-\alpha_1}{2}}\|x\|_{\HH^{\alpha_2}}.
\end{equation}
In addition, one has the following interpolation inequalities: if $\alpha_1<\alpha_2$ and $\alpha\in[\alpha_1,\alpha_2]$, for all $u\in H^{\alpha_2}$ and $x\in\HH^{\alpha_2}$, one has
\begin{equation}\label{eq:interp}
\|u\|_{H^\alpha}\le \|u\|_{H^{\alpha_1}}^{\frac{\alpha_2-\alpha}{\alpha_2-\alpha_1}}\|u\|_{H^{\alpha_2}}^{\frac{\alpha-\alpha_1}{\alpha_2-\alpha_1}}~,\quad \|x\|_{\HH^\alpha}\le \|x\|_{\HH^{\alpha_1}}^{\frac{\alpha_2-\alpha}{\alpha_2-\alpha_1}}\|x\|_{\HH^{\alpha_2}}^{\frac{\alpha-\alpha_1}{\alpha_2-\alpha_1}}.
\end{equation}

Let $\alpha\in\R$ be an arbitrary natural number. Define the projection operators $\Pi_u\colon \HH^\alpha\rightarrow H^\alpha$ and $\Pi_v\colon \HH^\alpha\rightarrow H^{\alpha-1}$ such that for all $(u,v)\in\HH^\alpha$ one has
\[
\Pi_u(u,v)=u\ \  \mbox{and} \ \ \Pi_v(u,v)=v.
\]
The operators $\Pi_u$ and $\Pi_v$ depend on $\alpha$, however the dependence is omitted to simplify notation. Note that for all $\alpha\in\R$ and all $x=(u,v)\in\HH^\alpha$, one has
\[
\max\bigl(\|\Pi_u x\|_{H^\alpha},\|\Pi_v x\|_{H^{\alpha-1}}\bigr)\le \|x\|_{\HH^\alpha}.
\]

\subsection{The damped wave equation semigroup}

Let $\gamma\in(0,\infty)$ be a positive real number. We introduce the linear operator $A_\gamma$ defined as follows: for all $\alpha\in\R$ and for all $x=(u,v)\in\HH^{\alpha+1}$, set
\[
A_\gamma x=\bigl(v,-\IL u-2\gamma v\bigr)\in \HH^{\alpha}.
\]
The linear operator $A_\gamma$ can be considered as an unbounded linear operator on $\HH$, with domain $\HH^1$.

The linear operator $A_\gamma$ generates a semigroup of linear operators $\bigl(e^{tA_\gamma}\bigr)_{t\ge 0}$ which can be defined as follows: for all $t\ge 0$ and all $x=(u,v)\in\HH$,
\begin{equation}\label{eq:semigroup}
e^{tA_\gamma}x=\sum_{n\in\N}\bigl(u_n(t)e_n,v_n(t)e_n\bigr)
\end{equation}
where for all $n\in\N$ the mappings $u_n,v_n:[0,\infty)\to\R$ are the solutions of the linear two-dimensional systems
\begin{equation}\label{eq:unvn}
\frac{d}{dt}\begin{pmatrix} u_n(t) \\ v_n(t) \end{pmatrix}=\begin{pmatrix} 0 & 1 \\ -\lambda_n & -2\gamma \end{pmatrix}\begin{pmatrix} u_n(t) \\ v_n(t) \end{pmatrix}=\begin{pmatrix} v_n(t) \\ -\lambda_n u_n(t) -2\gamma v_n(t) \end{pmatrix}~,~t\ge 0~;\quad \begin{pmatrix}u_n(0) \\ v_n(0)\end{pmatrix}=\begin{pmatrix}\langle u,e_n\rangle_H \\ \langle v,e_n\rangle_H \end{pmatrix}.
\end{equation}
Let us give the expressions of the solutions of the differential equations~\eqref{eq:unvn} above, for all $n\in\N$.
\begin{itemize}
\item If $\lambda_n-\gamma^2>0$, then for all $t\ge 0$ one has
\begin{equation}\label{eq:unvn1}
\left\lbrace
\begin{aligned}
u_n(t)&=e^{-\gamma t}\Bigl(\langle u,e_n\rangle_H~\cos(\sqrt{\lambda_n-\gamma^2} t)+\bigl(\gamma\langle u,e_n\rangle_H+\langle v,e_n\rangle_H\bigr)~\frac{\sin(\sqrt{\lambda_n-\gamma^2} t)}{\sqrt{\lambda_n-\gamma^2}}\Bigr)\\
v_n(t)&=e^{-\gamma t}\Bigl(\langle v,e_n\rangle_H~\cos(\sqrt{\lambda_n-\gamma^2} t)-\bigl(\lambda_n\langle u,e_n\rangle_H+\gamma\langle v,e_n\rangle_H\bigr)~\frac{\sin(\sqrt{\lambda_n-\gamma^2} t)}{\sqrt{\lambda_n-\gamma^2}}\Bigr).
\end{aligned}
\right.
\end{equation}
\item If $\lambda_n-\gamma^2=0$, then for all $t\ge 0$ one has
\begin{equation}\label{eq:unvn2}
\left\lbrace
\begin{aligned}
u_n(t)&=e^{-\gamma t}\Bigl(\langle u,e_n\rangle_H+\bigl(\gamma\langle u,e_n\rangle_H+\langle v,e_n\rangle_H\bigr)~t\Bigr)\\
v_n(t)&=e^{-\gamma t}\Bigl(\langle v,e_n\rangle_H-\bigl(\lambda_n\langle u,e_n\rangle_H+\gamma\langle v,e_n\rangle_H\bigr)~ t\Bigr).
\end{aligned}
\right.
\end{equation}
\item If $\lambda_n-\gamma^2<0$, then for all $t\ge 0$ one has
\begin{equation}\label{eq:unvn3}
\left\lbrace
\begin{aligned}
u_n(t)&=e^{-\gamma t}\Bigl(\langle u,e_n\rangle_H~\cosh(\sqrt{\gamma^2-\lambda_n} t)+\bigl(\gamma\langle u,e_n\rangle_H+\langle v,e_n\rangle_H\bigr)~\frac{\sinh(\sqrt{\gamma^2-\lambda_n} t)}{\sqrt{\gamma^2-\lambda_n}}\Bigr)\\
v_n(t)&=e^{-\gamma t}\Bigl(\langle v,e_n\rangle_H~\cosh(\sqrt{\gamma^2-\lambda_n} t)-\bigl(\lambda_n\langle u,e_n\rangle_H+\gamma\langle v,e_n\rangle_H\bigr)~\frac{\sinh(\sqrt{\gamma^2-\lambda_n} t)}{\sqrt{\gamma^2-\lambda_n}}\Bigr).
\end{aligned}
\right.
\end{equation}
\end{itemize}
The expressions above can be found using the following observation: the auxiliary mapping $w_n:t\mapsto e^{\gamma t}u_n(t)$ is solution of the linear second-order differential equation
\[
w_n''(t)+\bigl(\lambda_n-\gamma^2\bigr)w_n(t)=0,
\]
with initial values $w_n(0)=u_n(0)$ and $w_n'(0)=u_n'(0)+\gamma u_n(0)$. 

Note that for all $n\in\N$ and all $t\ge 0$, one has
\begin{equation}\label{eq:ineqdtunvn}
\frac12\frac{d}{dt}\Bigl(|u_n(t)|^2+\lambda_n^{-1}|v_n(t)|^2\Bigr)=-2\gamma \lambda_n^{-1}|v_n(t)|^2\le 0,
\end{equation}
which provides the following property for the semigroup $\bigl(e^{tA_\gamma}\bigr)_{t\ge 0}$: for all $\gamma\in(0,\infty)$, all $t\ge 0$ and all $x\in\HH$, one has
\begin{equation}\label{eq:boundsemigroup}
\|e^{tA_\gamma}x\|_\HH\le \|x\|_\HH.
\end{equation}
In order to study the long time behavior of exact and approximate solutions of the stochastic damped wave equation, the following property of the semigroup $\bigl(e^{tA_\gamma}\bigr)_{t\ge 0}$ is required.
\begin{lemma}\label{lem:semigroup}
Let $\gamma\in(0,\infty)$. There exist $\Mg\in(0,\infty)$ and $\rg\in(0,\infty)$ such that for all $t\in[0,\infty)$, all $\alpha\in\R$ and all $x\in\HH^\alpha$, one has
\begin{equation}\label{eq:expodecrease}
\|e^{tA_\gamma}x\|_{\HH^\alpha}\le \Mg e^{-\rg t}\|x\|_{\HH^\alpha}.
\end{equation}
\end{lemma}

The proof of Lemma~\ref{lem:semigroup}, is postponed to the appendix, see Section~\ref{sec:prooflemmasemigroup}. The techniques of the proof are similar to those used in~\cite[Lemma 5.1]{Salins:19}. Note that if the condition $\lambda_1>\gamma^2$ is satisfied, then it suffices to use the expression~\eqref{eq:unvn1} of the solution $(u_n(t),v_n(t))$ at time $t\ge 0$ of~\eqref{eq:unvn} to obtain a simpler proof of~\eqref{eq:expodecrease}. The proof provided in Section~\ref{sec:prooflemmasemigroup} covers the general case.

\begin{lemma}\label{lem:regulsemigroup}
For all $\alpha\in[0,1]$, there exists $C(\alpha,\gamma)\in(0,\infty)$ such that for all $t\in[0,1]$ and all $x\in\HH^\alpha$, one has
\begin{equation}\label{eq:regulsemigroup}
\|e^{tA_\gamma}x-x\|_{\HH}\le C(\alpha,\gamma)t^\alpha\|x\|_{\HH^\alpha}.
\end{equation}
\end{lemma}

\begin{proof}
Let $n\in\N$, and let $t\mapsto (u_n(t),v_n(t))$ denote the solution of~\eqref{eq:unvn}. On the one hand, the inequality
\[
|u_n(t)|^2+\lambda_n^{-1}|v_n(t)|^2\le |u_n(0)|^2+\lambda_n^{-1}|v_n(0)|^2,
\]
which follows from the inequality~\eqref{eq:ineqdtunvn} above, implies that for all $t\ge 0$ one has
\[
|u_n(t)-u_n(0)|^2+\lambda_n^{-1}|v_n(t)-v_n(0)|^2\le 4|u_n(0)|^2+4\lambda_n^{-1}|v_n(0)|^2.
\]
On the other hand, using the fundamental theorem of calculus and the upper bound above gives
\begin{align*}
|u_n(t)-u_n(0)|^2&=\Bigl(\int_0^t v_n(s)ds\Bigr)^2\le t^2\lambda_n\bigl(|u_n(0)|^2+\lambda_n^{-1}|v_n(0)|^2\bigr)\\
\lambda_n^{-1}|v_n(t)-v_n(0)|^2&=\lambda_n^{-1}\Bigl(\int_0^t \bigl(2\gamma v_n(s)+\lambda_nu_n(s)\bigr)ds\Bigr)^2\le C(\gamma)t^2\lambda_n\bigl(|u_n(0)|^2+\lambda_n^{-1}|v_n(0)|^2\bigr)
\end{align*}
for some positive real number $C(\gamma)\in(0,\infty)$.

Let $\alpha\in[0,1]$, then combining the two inequalities above, for all $t\ge 0$ and $n\in\N$ one has
\[
|u_n(t)-u_n(0)|^2+\lambda_n^{-1}|v_n(t)-v_n(0)|^2\le C(\alpha,\gamma)t^{2\alpha}\lambda_n^\alpha \bigl(|u_n(0)|^2+\lambda_n^{-1}|v_n(0)|^2\bigr),
\]
with $C(\alpha,\gamma)=4^{1-\alpha}C(\gamma)^\alpha$. It remains to see that for all $t\in[0,1]$ and all $x\in\HH^\alpha$ one has
\begin{align*}
\|e^{tA_\gamma}x-x\|^2_{\HH}&=\sum_{n\in\N}\Bigl(|u_n(t)-u_n(0)|^2+\lambda_n^{-1}|v_n(t)-v_n(0)|^2\Bigr)\\
&\le C(\alpha,\gamma)t^{2\alpha}\sum_{n\in\N}\lambda_n^\alpha \bigl(|u_n(0)|^2+\lambda_n^{-1}|v_n(0)|^2\bigr)\\
&=C(\alpha,\gamma)t^{2\alpha}\|x\|_{\HH^\alpha}^2,
\end{align*}
to conclude the proof.
\end{proof}

\subsection{Nonlinearity}

The nonlinearity is given by a mapping $f:u\in H\mapsto f(u)\in H$. Given such a function $f$, the mapping $F:\HH\to\HH^1$ is then defined as follows: for all $(u,v)\in\HH$,
\begin{equation}\label{eq:F}
F(u,v)=\bigl(0,f(u)\bigr).
\end{equation}
The mapping $f$ is assumed to be globally Lipschitz continuous, with Lipschitz constant $\Lf$ defined by
\[
\Lf=\underset{u_1,u_2\in H,u_1\neq u_2}\sup~\frac{\|f(u_2)-f(u_1)\|_H}{\|u_2-u_1\|_H}\in[0,\infty).
\]
Due to the definitions of the nonlinearity $F$ and of the norms $\|\cdot\|_\HH$ and $\|\cdot\|_{\HH^1}$, one then has the following inequalities: for all $x_1=(u_1,v_1),x_2=(u_2,v_2)\in \HH$,
\begin{align}
\|F(x_2)-F(x_1)\|_{\HH^{1}}&= \|f(u_2)-f(u_1)\|_{H} \le \Lf\|u_2-u_1\|_{H} \le\Lf\|x_2-x_1\|_{\HH}, \label{eq:LipF_HH1}\\
\|F(x_2)-F(x_1)\|_{\HH}&\le \lambda_1^{-\frac12}\|F(x_2)-F(x_1)\|_{\HH^{1}} \le \frac{\Lf}{\sqrt{\lambda_1}}\|x_2-x_1\|_{\HH}\label{eq:LipF_HH0}.
\end{align}

The inequalities~\eqref{eq:LipF_HH1} and~\eqref{eq:LipF_HH0} show that the mappings $F:\HH\to\HH^1$ and $F:\HH\to\HH$ are globally Lipschitz continuous.

The analysis of the long time behavior of the exact and approximate solutions of the considered stochastic damped wave equation is performed under the following condition which is assumed to be satisfied in the article.
\begin{ass}\label{ass:ergo}
Let $\gamma\in(0,\infty)$ be given, then the Lipschitz constant $\Lf$ of $f$ is assumed to satisfy the condition 
\begin{equation}\label{eq:ergo}
\Mg \Lf<\rg \sqrt{\lambda_1}
\end{equation}
where $\Mg,\rg\in(0,\infty)$ are given by Lemma~\ref{lem:semigroup}.
\end{ass}

In order to simplify notation in the sequel, it is convenient to set
\begin{equation}\label{eq:ergo2}
\rho(\gamma,f)=\rg-\frac{\Mg \Lf}{\sqrt{\lambda_1}}
\end{equation}
which is a positive real number owing to Assumption~\ref{ass:ergo}.

It remains to state the regularity conditions imposed on the nonlinearity $f$. Similar conditions are for instance considered in~\cite{MR3353942}, and can be found in the literature. They are satisfied for Nemytskii operators.
\begin{ass}\label{ass:F}
It is assumed that $f:H\to H^{-1}$ is of class $\mathcal{C}^2$ with bounded first and second derivatives.

In addition, it is assumed that for all $\alpha\in[0,1]$, there exists $C_\alpha\in(0,\infty)$ such that for all $u\in H^\alpha$ one has
\begin{equation}\label{eq:falpha}
\|f(u)\|_{H^\alpha}\le C_\alpha\bigl(1+\|u\|_{H^\alpha}\bigr),
\end{equation}
and it is assumed that for all $\alpha\in[0,1/2)$, there exists $C_\alpha\in(0,\infty)$ such that for all $u_1,u_2\in H^\alpha$ one has
\begin{equation}\label{eq:Dfalpha}
\|f(u_2)-f(u_1)\|_{H^{-1}}\le C_\alpha\bigl(1+\|u_1\|_{H^\alpha}+\|u_2\|_{H^{\alpha}}\bigr)\|u_2-u_1\|_{H^{-\alpha}}.
\end{equation}
\end{ass}

The condition~\eqref{eq:falpha} is useful to prove moment bounds in the $\|\cdot\|_{\HH^\alpha}$ norm (see Proposition~\ref{propo:momentbounds-exact}. The condition~\eqref{eq:Dfalpha} is useful in the proof of Theorem~\ref{theo:eem}.

Let us state some useful consequences of Assumption~\ref{ass:F}. First, the mapping $F:\HH\to \HH$ defined by~\eqref{eq:F} is of class $\mathcal{C}^2$ with bounded first and second derivatives: there exists $C_F\in(0,\infty)$ such that for all $x,h,k\in \HH$ one has
\begin{equation}
\|DF(x).h\|_\HH\le C_F\|h\|_\HH~,\quad \|D^2F(x).(h,k)\|_\HH\le C_F\|h\|_\HH \|k\|_\HH.
\end{equation}
Moreover, the condition~\eqref{eq:falpha} implies the following properties:
\begin{itemize}
\item if $\alpha\in[1,2]$, for all $x\in\HH^\alpha$ one has
\begin{equation}\label{eq:Falpha}
\|F(x)\|_{\HH^\alpha}\le C_\alpha\bigl(1+\|x\|_{\HH^{\alpha/2}}\bigr).
\end{equation}
\item if $\alpha\in[0,1]$, for all $x\in\HH$ one has
\begin{equation}\label{eq:Falphabis}
\|F(x)\|_{\HH^\alpha}\le C_0\bigl(1+\|x\|_{\HH}\bigr).
\end{equation}
\end{itemize}
The proof of the inequality~\eqref{eq:Falpha} is straightforward: if $x=(u,v)\in\HH^\alpha$, using the inequality~\eqref{eq:falpha} one has
\[
\|F(x)\|_{\HH^\alpha}=\|f(u)\|_{H^{\alpha-1}}\le C_{\alpha}\bigl(1+\|u\|_{H^{\alpha-1}}\bigr)\le C_\alpha\bigl(1+\|x\|_{\HH^{\alpha-1}}\bigr)\le C_\alpha\bigl(1+\|x\|_{\HH^{\alpha/2}}\bigr),
\]
using the inequality $\alpha-1\le \alpha/2$ in the last step. In addition, to prove the inequality~\eqref{eq:Falphabis}, it suffices to observe that for all $x\in \HH$ one has $\|F(x)\|_{\HH^\alpha}\le \|F(x)\|_{\HH^1}\le \|f(u)\|_H$.

Finally, the condition~\eqref{eq:Dfalpha} from Assumption~\ref{ass:F} implies the following result: if $\beta\in[\frac12,1]$, for all $x\in \HH^\beta$, one has
\begin{equation}\label{eq:Dfalpha2}
\|f(u_2)-f(u_1)\|_{H^{-1}}\le C_\beta\bigl(1+\|u_1\|_{H^\beta}+\|u_2\|_{H^{\beta}}\bigr)\|u_2-u_1\|_{H^{\beta-1}}.
\end{equation}
That result follows from applying~\eqref{eq:Dfalpha} with $\alpha=1-\beta$, noting that $\beta\ge 1-\beta$ and using the inequality~\eqref{eq:comparenorms}.

\subsection{The Wiener process}

Let $\bigl(\Omega,\mathcal{F},\mathbb{P})$ be a probability space, where the expectation operator is denoted by $\E[\cdot]$. A filtration $\bigl(\mathcal{F}_t\bigr)_{t\ge 0}$ satisfying the usual conditions is considered.

The stochastic evolution equation considered in this paper is driven by a $\HH$-valued $Q$-Wiener process $\bigl(\IW^Q(t)\bigr)_{t\ge 0}$, adapted to the filtration, which has the expression
\begin{equation}\label{eq:IWQ}
\IW^Q(t)=\bigl(0,W^Q(t)\bigr),\quad \forall~t\ge 0,
\end{equation}
where $\bigl(W^Q(t)\bigr)_{t\ge 0}$ is a $Q$-Wiener process taking values in $H^{-1}$, which can be described as follows: there exists a sequence $\bigl(\beta_n\bigr)_{n\in\N}$ of independent standard real-valued Wiener processes, adapted to the filtration $\bigl(\mathcal{F}_t\bigr)_{t\ge 0}$ a sequence $\bigl(q_n\bigr)_{n\in\N}$ of nonnegative real numbers and a complete orthonormal system $\bigl(e_n^Q\bigr)_{n\in\N}$ of $H$, such that for all $t\ge 0$ one has
\begin{equation}\label{eq:WQ}
W^Q(t)=\sum_{n\in\N}\sqrt{q_n}\beta_n(t)e_n^Q,\quad \forall~t\ge 0.
\end{equation}
To ensure that the Wiener process $\bigl(W^Q(t)\bigr)_{t\ge 0}$ takes values in $H^{-1}$, and therefore that the Wiener process $\bigl(\IW^Q(t)\bigr)_{t\ge 0}$ takes values in $\mathcal{H}$, the condition
\begin{equation}\label{eq:conditionQ}
\sum_{n\in\N}q_n\|e_n^Q\|_{H^{-1}}^2<\infty
\end{equation}
is imposed. Introduce the linear operators $Q$ and $Q^{1/2}$, defined by
\begin{align*}
Qh&=\sum_{n\in\N} q_n \langle h,e_n^Q\rangle_{H}e_n^Q,\\
Q^{1/2}h&=\sum_{n\in\N} \sqrt{q_n} \langle h,e_n^Q\rangle_{H}e_n^Q.
\end{align*}
The condition~\eqref{eq:conditionQ} above means that $Q^{1/2}$ is an Hilbert--Schmidt linear operator from $H$ to $H^{-1}$: one has
\[
\|Q^{1/2}\|_{\mathcal{L}_2(H,H^{-1})}^2=\sum_{n\in\N}\|Q^{1/2}e_n^Q\|_{H^{-1}}^2=\sum_{n\in\N}q_n\|e_n^Q\|_{H^{-1}}^2<\infty.
\]

In the sequel the following condition, which is stronger than~\eqref{eq:conditionQ}, is imposed.
\begin{ass}\label{ass:Q}
It is assumed that there exists $\beta\in(0,1]$ such that
\begin{equation}\label{eq:Qbeta}
\|Q^{1/2}\|_{\mathcal{L}_2(H,H^{\beta-1})}^2=\sum_{n\in\N}\|Q^{1/2}e_n^Q\|_{H^{\beta-1}}^2=\sum_{n\in\N}q_n\|e_n^Q\|_{H^{\beta-1}}^2<\infty.
\end{equation}
\end{ass}

If Assumption~\ref{ass:Q} holds, and if $\beta\in(0,1]$ is such that the condition~\eqref{eq:Qbeta} holds, then the Wiener process $\bigl(\IW^Q(t)\bigr)_{t\ge 0}$ takes values in $\HH^\beta$. In this case, for all $p\in[1,\infty)$ there exists $C_{\beta,p}\in(0,\infty)$ such that for all $t,s\ge 0$ one has
\begin{equation}\label{eq:incrementsWQ}
\E[\|\IW^Q(t)-\IW^Q(s)\|_{\HH^\beta}^p]=C_{\beta,p}(t-s)^{\frac{p}{2}}.
\end{equation}

If the complete orthonormal systems $\bigl(e_n\bigr)_{n\in\N}$ and $\bigl(e_n^Q\bigr)_{n\in\N}$ of $H$ coincide, then the linear operator $\IL$ and the covariance operator $\IL$ commute. In the sequel, this situation is referred to as the \emph{commutative noise case}. On the contrary, the expression \emph{non-commutative noise case} is used if the commutativity assumption does not hold. Many arguments in the analysis hold in both cases, however it is sometimes necessary to treat the two cases separately. In fact, the orders of convergence obtained below are higher in the commutative noise case. A parameter $\theta$ is introduced in the statements and proofs below: we set $\theta=1$ in the commutative noise case, and $\theta=1/2$ in the non-commutative noise case.

Let us give examples of covariance operators $Q$ and give the admissible values of $\beta$ in Assumption~\ref{ass:Q} for these examples. First, if $Q$ is a trace-class linear operator from $H$ to $H$, then the condition~\eqref{eq:Qbeta} is satisfied for $\beta=1$ (and for all $\beta\in[0,1]$ as a consequence of~\eqref{eq:comparenorms}). Second, if $Q$ is the identity operator, the choice of the complete orthonormal system $\bigl(e_n^Q\bigr)_{n\in\N}$ is arbitrariy so one can choose $e_n^Q=e_n$ for all $n\in\N$, and one has the commutative noise situation.  In that case the condition~\eqref{eq:Qbeta} is satisfied if and only if $\sum_{n\in\N}\lambda_n^{\beta-1}<\infty$. On the one hand, if $d=1$, then one has $\lambda_n\sim {\rm c}_dn^2$ and the condition~\eqref{eq:Qbeta} is satisfied if and only if $\beta<1/2$. On the other hand, if $d\ge 2$, then one has $\lambda_n\sim {\rm c}_dn^{2/d}$ and the condition~\eqref{eq:Qbeta} is not satisfied, for any $\beta\ge 0$.

We refer to for instance~\cite[Chapter~4]{DPZ} for the theory of stochastic integration in Hilbert spaces.

\section{The stochastic damped wave equation}\label{sec:spde}

In this work we consider the stochastic evolution equation
\begin{equation}\label{eq:spde}
\left\lbrace
     \begin{aligned}
&du(t)=v(t)dt,\\
&dv(t)=(-\IL u(t)-2\gamma v(t))\,dt+f(u(t))\,dt+\,dW^Q(t),\\
&u(0)=u_0,\  \  v(0)=v_0, 
     \end{aligned}
\right.
\end{equation}
where the unknowns $\bigl(u(t)\bigr)_{t\ge 0}$ and $\bigl(v(t)\bigr)_{t\ge 0}$ are $H$ and $H^{-1}$-valued stochastic processes respectively. In the sequel, we consider the equivalent formulation
\begin{equation}\label{eq:see}
\left\lbrace
     \begin{aligned}
&dX(t)=A_\gamma X(t)dt+F(X(t))dt+d\IW^Q(t),\\
&X(0)=x_0,  
     \end{aligned}
\right.
\end{equation}
where the unknown $\bigl(X(t)\bigr)_{t\ge 0}$ is a $\HH$-valued stochastic processes, and where $X(t)=(u(t),v(t))$ for all $t\ge 0$. The linear operator $A_\gamma$, the nonlinearity $F$ and the Wiener process $\IW^Q$ have been introduced in Section~\ref{sec:setting}. The initial value $x_0=(u_0,v_0)$ is a $\mathcal{F}_0$-measurable $\HH$-valued Gaussian random variable.

Let us recall that a $\HH$-valued stochastic process $\bigl(X(t)\bigr)_{t\ge 0}$ (adapted to the filtration $\bigl(\mathcal{F}_t\bigr)_{t\ge 0}$ and with continuous trajectories) is a mild solution of the stochastic evolution equation~\eqref{eq:see} with initial value $X(0)=x_0$, if for all $t\ge 0$ almost surely one has
\begin{equation}\label{eq:mild}
X(t)=e^{tA_{\gamma}}x_0+\int_0^te^{(t-s)A_{\gamma}}F(X(s))\,ds+\int_0^te^{(t-s)A_{\gamma}}\,d\IW^Q(s),
\end{equation}
where the semigroup $\bigl(e^{tA_\gamma}\bigr)_{t\ge 0}$ is given by~\eqref{eq:semigroup}. In order to study well-posedness, regularity properties and long-time behavior of mild solutions of~\eqref{eq:see}, it is first necessary to analyze the properties of the stochastic convolution, defined by
\begin{equation}\label{eq:convolution}
\IZ(t)=\int_0^te^{(t-s)A_{\gamma}}\,d\IW^Q(s),\ \  t\geq 0.
\end{equation}
One has the following result.
\begin{lemma}\label{lem:convolution}
Let Assumption~\ref{ass:Q} be satisfied. For any $\beta\in(0,1]$ such that the condition~\eqref{eq:Qbeta} holds, for all $p\in[1,\infty)$, one has
\begin{equation}\label{eq:momentboundsconvolution}
\underset{t\ge 0}\sup~\E[\|\IZ(t)\|_{\HH^\beta}^{p}]<\infty.
\end{equation} 
\end{lemma}

\begin{proof}
For any $t\ge 0$, $\IZ(t)$ is a Gaussian random variable, therefore it suffices to prove the moment bounds~\eqref{eq:momentboundsconvolution} for $p=2$. Using It\^o's isometry and the inequality~\eqref{eq:expodecrease} from Lemma~\ref{lem:semigroup}, one has
\begin{align*}
\E[\|\IZ(t)\|_{\HH^\beta}^2]&=\int_{0}^{t}\sum_{n\in\N}q_n\|e^{(t-s)A_\gamma}\bigl(0,e_n^Q\bigr)\|_{\HH^\beta}^2 ds\\
&\le (\Mg)^2\int_0^t e^{-2\rg(t-s)}ds \sum_{n\in\N}q_n\|e_n^Q\|_{H^{\beta-1}}^2\\
&\le \frac{(\Mg)^2}{2\rg}\|Q^{1/2}\|_{\mathcal{L}_2(H,H^{\beta-1})}^2<\infty,
\end{align*}
using the assumption that the condition~\eqref{eq:Qbeta} is satisfied in the last step. This concludes the proof.
\end{proof}

It is then straightforward to prove that for any $\mathcal{F}_0$-measurable $\HH$-valued initial value $x_0$, there exists a unique global mild solution~\eqref{eq:mild} of the stochastic evolution equation~\eqref{eq:see}, which can be constructed using a standard fixed point procedure (note that the nonlinearity $F$ from $\HH$ to $\HH$ is globally Lipschitz continuous, see~\eqref{eq:LipF_HH0}). The details of this standard proof are omitted. In addition, $X(t)$ takes values in $\HH^\beta$ for $\beta\in[0,1]$, if $x_0\in\HH^\beta$ and if the condition~\eqref{eq:Qbeta} is satisfied.

We now focus on proving two properties of the mild solutions of~\eqref{eq:see}. We first study moment bounds which are uniform in time, and second we show the existence and uniqueness of an invariant distribution denoted by $\mu_\infty$. These two properties require Assumption~\ref{ass:ergo} to be satisfied.

\begin{propo}\label{propo:momentbounds-exact}
Let Assumptions~\ref{ass:ergo} and~\ref{ass:Q} be satisfied. Let $p\in[1,\infty)$. Let $\beta\in[0,1]$, such that the condition~\eqref{eq:Qbeta} is satisfied. There exists $C_{\beta,p}\in(0,\infty)$ such that if the initial condition $X(0)=x_0$ is a $\mathcal{F}_0$-measurable random variable, which satisfies $\E[\|x_0\|_{\HH^\beta}^{p}]<\infty$, then one has
\begin{equation}\label{eq:momentbounds-exact}
\underset{t\ge 0}\sup~\E[\|X(t)\|_{\HH^\beta}^{p}]\le C_{\beta,p}\Bigl(1+\E[\|x_0\|_{\HH^\beta}^{p}]\Bigr).
\end{equation}
\end{propo}

\begin{proof}
Introduce the auxiliary process $\bigl(Y(t)\bigr)_{t\ge 0}$ defined by
\[
Y(t)=X(t)-\IZ(t),
\]
where $\bigl(X(t)\bigr)_{t\ge 0}$ is the mild solution of~\eqref{eq:see} and $\bigl(\IZ(t)\bigr)_{t\ge 0}$ is the stochastic convolution given by~\eqref{eq:convolution}. In order to prove~\eqref{eq:momentbounds-exact} for $\beta=0$, owing to Lemma~\ref{lem:convolution}, it suffices to check that
\[
\underset{t\ge 0}\sup~\E[\|Y(t)\|_{\HH}^{p}]\le C_{\beta,p}\Bigl(1+\E[\|x_0\|_{\HH}^{p}]\Bigr).
\]
For all $t\ge 0$, one has
\[
Y(t)=e^{tA_{\gamma}}x_0+\int_0^te^{(t-s)A_{\gamma}}F(Y(s)+\IZ(s))\,ds.
\]
Using the inequality~\eqref{eq:expodecrease} from Lemma~\ref{lem:semigroup} and the Lipschitz continuity property~\eqref{eq:LipF_HH0} of $F$ from $\HH$ to $\HH$, one obtains, for all $t\ge 0$,
\begin{align*}
\|Y(t)\|_\HH\le \Mg e^{-\rg t}\|x_0\|_\HH+\int_0^t \frac{\Lf \Mg}{\sqrt{\lambda_1}}e^{-\rg(t-s)}\|Y(s)\|_\HH ds+\int_0^t \Mg e^{-\rg(t-s)}\|F(\IZ(s))\|_\HH ds.
\end{align*}
Using Minkowskii's inequality and the moment bounds~\eqref{eq:momentboundsconvolution} from Lemma~\ref{lem:convolution} for the stochastic convolution $\IZ(t)$, for any $p\in[1,\infty)$ there exists $C_p\in(0,\infty)$ such that one has for all $t\ge 0$
\begin{align*}
\bigl(\E[\|Y(t)\|_{\HH}^{p}]\bigr)^{\frac1p}&\le C_p\Bigl(1+\bigl(\E[\|x_0\|_{\HH}^{p}]\bigr)^{\frac1p}\Bigr)+\int_0^t \frac{\Lf \Mg}{\sqrt{\lambda_1}}e^{-\rg(t-s)}\bigl(\E[\|Y(s)\|_\HH^p]\bigr)^{\frac1p} ds.
\end{align*}
The condition~\eqref{eq:ergo} ensures that one has $\frac{\Lf \Mg}{\sqrt{\lambda_1}}<\rg$. As a consequence, applying the Gronwall inequality to the mapping $t\mapsto e^{\rg t}\bigl(\E[\|Y(t)\|_{\HH}^{p}]\bigr)^{\frac1p}$ one obtains
\[
\underset{t\ge 0}\sup~\bigl(\E[\|Y(t)\|_{\HH}^{p}]\bigr)^{\frac1p}\le C_p\Bigl(1+\bigl(\E[\|x_0\|_{\HH}^{p}]\bigr)^{\frac1p}\Bigr),
\]
which proves the inequality~\eqref{eq:momentbounds-exact} for $\beta=0$.

Let now $\beta\in(0,1]$ and assume that the condition~\eqref{eq:Qbeta} holds, and that the initial value $x_0$ satisfies $\E[\|x_0\|_{\HH^\beta}^{p}]<\infty$. Since $F$ is globally Lipschitz continuous from $\HH$ to $\HH^1$, see~\eqref{eq:LipF_HH1}, it is also globally Lipschitz continuous from $\HH$ to $\HH^\beta$ owing to the inequality~\eqref{eq:comparenorms}. Therefore there exists $C_\beta\in(0,\infty)$ such that, using the mild formulation~\eqref{eq:mild}, for all $t\ge 0$ one has
\[
\|X(t)\|_{\HH^\beta}\le \|x_0\|_{\HH^\beta}+C_\beta\int_0^t e^{-\rg(t-s)}\bigl(1+\|X(s)\|_{\HH}\bigr)ds+\|\IZ(t)\|_{\HH^\beta}.
\]
Using Minkowskii's inequality, the moment bounds for $\beta=0$ established above and the moment bounds~\eqref{eq:momentboundsconvolution} from Lemma~\ref{lem:convolution} for the stochastic convolution $\IZ(t)$, one obtains 
\[
\underset{t\ge 0}\sup~\bigl(\E[\|X(t)\|_{\HH^\beta}^{p}]\bigr)^{\frac1p}\le C_{\beta,p}\Bigl(1+\bigl(\E[\|x_0\|_{\HH^\beta}^{p}]\bigr)^{\frac1p}\Bigr).
\]
The proof of the moment bounds~\eqref{eq:momentbounds-exact} is thus completed.
\end{proof}

\begin{propo}\label{propo:ergo-exact}
Let Assumptions~\ref{ass:ergo} and~\ref{ass:Q} be satisfied. The $\HH$-valued Markov process $\bigl(X(t)\bigr)_{t\ge 0}$ admits a unique invariant probability distribution $\mu_\infty$, which satisfies the following properties.
\begin{itemize}
\item If $x_0$ has distribution $\mu_\infty$, then for all $t\ge 0$ $X(t)$ has distribution $\mu_\infty$.
\item For all $\beta\in[0,1]$ such that the condition~\eqref{eq:Qbeta} is satisfied and all $p\in[1,\infty)$, one has
\begin{equation}\label{eq:moment-mu-exact}
\int \|x\|_{\HH^\beta}^{2p}d\mu_\infty(x)<\infty.
\end{equation}
\item There exists $C\in(0,\infty)$, such that if $\varphi:\HH\to\R$ is a Lipschitz continuous mapping, for all $t\ge 0$ one has
\begin{equation}\label{eq:ergo-exact}
\big|\E[\varphi(X(t))]-\int\varphi d\mu_\infty\big|\le C{\rm Lip}(\varphi)e^{-\rho(\gamma,f)t}\bigl(1+\E[\|X(0)\|_{\HH}]\bigr),
\end{equation}
where $\rho(\gamma,f)>0$ is defined by~\eqref{eq:ergo2} and ${\rm Lip}(\varphi)=\underset{x_1,x_2\in \HH, x_1\neq x_2}\sup~\frac{|\varphi(x_2)-\varphi(x_1)|}{\|x_2-x_1\|_H}$.
\end{itemize}
\end{propo}

\begin{proof}
Let $x_{0,1}\in\HH$ and $x_{0,2}\in\HH$ be two arbitrary initial values, and denote by $\bigl(X_1(t)\bigr)_{t\ge 0}$ and $\bigl(X_2(t)\bigr)_{t\ge 0}$ the mild solutions given by~\eqref{eq:mild} of the stochastic evolution equation~\eqref{eq:see}, driven by the same Wiener process $\bigl(\IW^Q(t)\bigr)_{t\ge 0}$, and with initial values $X_1(0)=x_{0,1}$ and $X_2(0)=x_{0,2}$ respectively.

We claim that for all $t\ge 0$ one has
\begin{equation}\label{eq:claim_ergo}
\|X_2(t)-X_1(t)\|_{\HH}\le e^{-\rho(\gamma,f)t}\|x_{0,2}-x_{0,1}\|_{\HH}
\end{equation}
where the positive real number $\rho(\gamma,f)>0$ is defined by~\eqref{eq:ergo2}. All the items of Proposition~\ref{propo:ergo-exact} can then be obtained using standard arguments, the details are omitted.

The proof of the claim~\eqref{eq:claim_ergo} is straightforward. Let $r(t)=\|X_2(t)-X_1(t)\|_{\HH}$ for all $t\ge 0$. Using the mild formulation~\eqref{eq:mild}, the inequality~\eqref{eq:expodecrease} from Lemma~\ref{lem:semigroup} and the Lipschitz continuity property~\eqref{eq:LipF_HH0} of $F$ from $\HH$ to $\HH$, for all $t\ge 0$ one obtains
\begin{align*}
r(t)&\le \Mg e^{-\rg t}\|x_{0,2}-x_{0,1}\|_{\HH}+\int_{0}^{t}\Mg e^{-\rg (t-s)}\|F(X_2(s))-F(X_1(s))\|_{\HH}ds\\
&\le \Mg e^{-\rg t}r(0)+\int_{0}^{t}\frac{\Mg\Lf}{\sqrt{\lambda_1}} e^{-\rg(t-s)}\|r(s)\|_{\HH}ds.
\end{align*}
Applying the Gronwall inequality to the mapping $t\mapsto e^{\rg t}\|r(t)\|_\HH$ and using the condition~\eqref{eq:ergo} from Assumption~\ref{ass:ergo} then yields the claim~\eqref{eq:claim_ergo}.

This concludes the proof of Proposition~\ref{propo:ergo-exact}.
\end{proof}

\section{Numerical methods and convergence results}\label{sec:numerics}

\subsection{Spatial discretization: spectral Galerkin method}\label{sec:galerkin}

For any integer $N\in\N$, introduce the finite dimensional spaces
\begin{equation}
H_N={\rm span}\left\{e_1,\ldots,e_N\right\}~,\quad \HH_N={\rm span}\left\{(e_1,0),(0,e_1),\ldots,(e_N,0),(0,e_N)\right\}.
\end{equation}
Observe that one has $\|(0,e_n)\|_\HH=\|e_n\|_{H^{-1}}=\lambda_n^{-\frac12}$, therefore $(e_1,0),(0,\sqrt{\lambda_1}e_1),\ldots,(e_N,0),(0,\sqrt{\lambda_N}e_N)$ is an orthonormal system of the finite dimensional space $\HH_N$.

In addition, introduce the associated orthogonal projection operators denoted by $P_N$ and $\IP_N$: for all $u=\sum_{n\in\N}\langle u,e_n\rangle_H e_n\in \HH^{-1}$, one has
\begin{equation}
P_Nu=\sum_{n=1}^{N}\langle u,e_n\rangle_H e_n
\end{equation}
and for all $x=(u,v)\in\HH$, one has
\begin{equation}
\IP_Nx=\bigl(P_Nu,P_Nv\bigr)=\Bigl(\sum_{n=1}^{N}\langle u,e_n\rangle_H e_n,\sum_{n=1}^{N}\langle v,e_n\rangle_H e_n\Bigr)=\sum_{n=1}^{N}\langle x,(e_n,0)\rangle_{\HH}(e_n,0)+\sum_{n=1}^{N}\lambda_n\langle x,(0,e_n)\rangle_{\HH}(0,e_n).
\end{equation}

One has the following result: for all $\alpha\in[0,\infty)$, for all $x\in\HH^\alpha$ and $N\in\N$, one has
\begin{equation}\label{eq:errorPN}
\|\IP_Nx-x\|_{\HH}\le \lambda_N^{-\frac{\alpha}{2}}\|x\|_{\HH^\alpha}.
\end{equation}

Note that the unbounded linear operator $A_\gamma$ semigroup $\bigl(e^{tA_\gamma}\bigr)_{t\ge 0}$ given by~\eqref{eq:semigroup} commute with the projection operators $\IP_N$, for all $N\in\N$. In addition, the range of $\IP_N$ is a finite dimensional subspace of $\HH^\alpha$ for all $N\in\N$. Finally, it is worth mentioning that for all $\alpha\in\R$, one has
\[
\underset{N\in\N}\sup~\underset{x\in\HH^\alpha\setminus\{0\}}\sup~\frac{\|\IP_N x\|_{\HH^\alpha}}{\|x\|_{\HH^\alpha}}=1.
\]

For all $N\in\N$, the mapping $F_N:\HH_N\to\HH_N$ is defined by the expression $F_N(x)=\IP_NF(x)$ for all $x\in\HH_N$. Define also the mapping $f_N:H_N\to H_N$ by $f_N(u)=P_Nf(u)$ for all $u\in H_N$.

Using the spectral Galerkin approximation consists in approximating the mild solution $\bigl(X(t)\bigr)_{t\ge 0}$ of~\eqref{eq:see} by the $\HH_N$-valued process $\bigl(X^N(t)\bigr)_{t\ge 0}$ which is the solution of
\begin{equation}\label{eq:mildgalerkin}
X^N(t)=e^{tA_{\gamma}}\IP_Nx_0+\int_0^te^{(t-s)A_{\gamma}}F_N(X^N(s))\,ds+\int_0^te^{(t-s)A_{\gamma}}\IP_N\,d\IW^Q(s).
\end{equation}
Note that $\int_0^te^{(t-s)A_{\gamma}}\IP_N\,d\IW^Q(s)=\IP_N \IZ(t)$ where the stochastic convolution $\IZ(t)$ is defined by~\eqref{eq:convolution}.

Observe that~\eqref{eq:mildgalerkin} is the mild formulation associated with the stochastic evolution equation
\begin{equation}\label{eq:seegalerkin}
\left\lbrace
     \begin{aligned}
&dX^N(t)=A_\gamma X^N(t)dt+F_N(X^N(t))dt+\IP_Nd\IW^Q(t),\\
&X^N(0)=x_0^N=\IP_Nx_0. 
     \end{aligned}
\right.
\end{equation}
Define $u^N(t)=\Pi_uX^N(t)$ and $v^N(t)=\Pi_vX^N(t)$ for all $t\ge 0$. Then the stochastic evolution equation~\eqref{eq:seegalerkin} can be ,\begin{equation}\label{eq:spdegalerkin}
\left\lbrace
     \begin{aligned}
&du^N(t)=v^N(t)dt,\\
&dv^N(t)=(-\IL u^N(t)-\gamma v^N(t))\,dt+f_N(u^N(t))\,dt+P_N\,dW^Q(t),\\
&u^N(0)=P_Nu_0,\  \  v^N(0)=P_Nv_0.
     \end{aligned}
\right.
\end{equation}

It is straightforward to check that for all $N\in\N$, and for any initial value $x_0\in\HH$, there exists a unique global mild solution $\bigl(X^N(t)\bigr)_{t\ge 0}$ of~\eqref{eq:seegalerkin}. One has the following versions of Proposition~\ref{propo:ergo-exact} and~\ref{propo:momentbounds-exact}, where it is worth mentioning that the constants do not depend on the spatial discretization parameter $N$. The proofs are omitted since they use the same arguments as above for the exact solution.

\begin{propo}\label{propo:momentbounds-galerkin}
Let Assumptions~\ref{ass:ergo} and~\ref{ass:Q} be satisfied. Let $p\in[1,\infty)$. Let $\beta\in[0,1]$, such that the condition~\eqref{eq:Qbeta} is satisfied. There exists $C_{\beta,p}\in(0,\infty)$ such that if the initial condition $x_0$ is a $\mathcal{F}_0$-measurable random variable, which satisfies $\E[\|x_0\|_{\HH^\beta}^{p}]<\infty$, then one has
\begin{equation}\label{eq:momentbounds-galerkin}
\underset{N\in\N}\sup~\underset{t\ge 0}\sup~\E[\|X^N(t)\|_{\HH^\beta}^{p}]\le C_{\beta,p}\Bigl(1+\E[\|x_0\|_{\HH^\beta}^{p}]\Bigr).
\end{equation}
\end{propo}

\begin{propo}\label{propo:ergo-galerkin}
Let Assumptions~\ref{ass:ergo} and~\ref{ass:Q} be satisfied. For any $N\in\N$, the $\HH_N$-valued Markov process $\bigl(X^N(t)\bigr)_{t\ge 0}$ admits a unique invariant probability distribution $\mu_\infty^N$, which satisfies the following properties.
\begin{itemize}
\item If $x_0^N$ has distribution $\mu_\infty^N$, then for all $t\ge 0$ $X^N(t)$ has distribution $\mu_\infty^N$.
\item For all $\beta\in[0,1]$ such that the condition~\eqref{eq:Qbeta} is satisfied and all $p\in[1,\infty)$, one has
\begin{equation}\label{eq:moment-mu-galerkin}
\underset{N\in\N}\sup~\int \|x\|_{\HH^\beta}^{2p}d\mu_\infty^N(x)<\infty.
\end{equation}
\item There exists $C\in(0,\infty)$, such that if $\varphi:\HH\to\R$ is a Lipschitz continuous mapping, for all $N\in\N$ and all $t\ge 0$ one has
\begin{equation}\label{eq:ergo-galerkin}
\big|\E[\varphi(X^N(t))]-\int\varphi d\mu_\infty^N\big|\le C{\rm Lip}(\varphi)e^{-\rho(\gamma,f)t}\bigl(1+\E[\|X^N(0)\|_{\HH}]\bigr),
\end{equation}
where $\rho(\gamma,f)>0$ is defined by~\eqref{eq:ergo2} and ${\rm Lip}(\varphi)=\underset{x_1,x_2\in \HH, x_1\neq x_2}\sup~\frac{|\varphi(x_2)-\varphi(x_1)|}{\|x_2-x_1\|_H}$.
\end{itemize}
\end{propo}

Using the inequality~\eqref{eq:errorPN} and standard arguments, one has the following convergence result.
\begin{lemma}\label{lem:strongerror_galerkin}
Let $T\in(0,\infty)$, $\beta\in[0,1]$ such that the condition~\eqref{eq:Qbeta} from Assumption~\ref{ass:Q} holds, and let $x_0$ be a $\mathcal{F}_0$-measurable random variable such that $\E[\|x_0\|_{\HH^\beta}]<\infty$. Then there exists $C_\beta(T)\in(0,\infty)$ such that
\[
\E[\|X(T)-X^N(T)\|_\HH]\le C_\beta(T)\lambda_N^{-\frac{\beta}{2}}\bigl(1+\E[\|x_0\|_{\HH^\beta}]\bigr).
\]
\end{lemma}
The proof is omitted.

\subsection{Fully discrete scheme: exponential Euler method}\label{sec:fully}

The time-step size is denoted by $\tau$, and without loss of generality it is assumed below that $\tau\in(0,1)$. For any nonnegative integer $m\ge 0$, set $t_m=m\tau$ and introduce the increments of the Wiener process, denoted by $\Delta \IW^Q_m=\IW^Q(t_{m+1})-\IW^Q(t_m)$. For all $t\ge 0$,  set $\ell(t)=\lfloor t/\tau\rfloor$, where $\lfloor\cdot\rfloor$ denotes the integer part function: if $m\ge 0$ and $t\in[t_m,t_{m+1})$, one has $\ell(t)=t_m$.

Applying an exponential Euler method, the following approximation of the solution $(X^N(t))_{t\geq0}$ of the finite dimensional stochastic evolution equation~\eqref{eq:mildgalerkin} is obtained: set the initial value $X^N_0=X^N(0)=\IP_Nx_0$ and for all $m\ge 0$, set
\begin{equation}\label{eq:expoEuler}
X^N_{m+1}=e^{\tau A_{\gamma}}\Bigl(X^N_m+\tau F_N(X^N_m)+\IP_N\Delta \IW^Q_m\Bigr).
\end{equation}
For all $N\in\N$ and all $m\ge 0$, define $u^N_{m}=\Pi_u X^N_{m}$ and $v^N_{m}=\Pi_v X^N_m$. The random variable $X^N_m$ should be interpreted as an approximation of $X^N(t_m)$.

It is straightforward to check that the solution $\bigl(X_m^{N}\bigr)_{m\ge 0}$ of the numerical scheme~\eqref{eq:expoEuler} can be written as follows: for all $m\ge 0$ one has
\begin{equation}\label{eq:mildexpoEuler}
X^N_m=e^{t_m A_{\gamma}}X^N_0+\tau\sum_{\ell=0}^{m-1} e^{(t_m-t_\ell) A_\gamma}F_N(X^N_\ell)+\sum_{\ell=0}^{m-1} e^{(t_m-t_\ell)A_\gamma} \IP_N\Delta \IW^Q_\ell.
\end{equation}
The expression~\eqref{eq:mildexpoEuler} above is a mild formulation for the solution of~\eqref{eq:expoEuler}, which is a discrete time version of the continuous time mild formulations~\eqref{eq:mildgalerkin} and~\eqref{eq:mild} associated with~\eqref{eq:seegalerkin} and~\eqref{eq:see} respectively.

It is convenient to introduce notation for the discrete time version of the stochastic convolution (see~\eqref{eq:convolution} in the continuous time case): for all $N\in\N$ and $m\ge 0$ set
\begin{equation}\label{eq:discreteconvolution}
	\IZ^N_m:=\sum_{\ell=0}^{m-1} e^{(t_m-t_\ell)A_\gamma} \IP_N\Delta \IW^Q_\ell=\int_0^{t_m}e^{(t_m-t_{\ell(s)})A_\gamma}\IP_N\,d\IW^Q(s).
\end{equation}
One obtains the following result, using the same tools as in the proof of Lemma~\ref{lem:convolution}.
\begin{lemma}\label{lem:discreteconvolution}
Let Assumption~\ref{ass:Q} be satisfied. For any $\beta\in(0,1]$ such that the condition~\eqref{eq:Qbeta} holds, for all $p\in[1,\infty)$, one has
\begin{equation}\label{eq:momentboundsdiscreteconvolution}
\underset{N\in\N}\sup~\underset{\tau\in(0,1)}\sup~\underset{m\ge 0}\sup~\E[\|\IZ_m^N\|_{\HH^\beta}^{p}]<\infty.
\end{equation} 
\end{lemma}

Let us then state and prove moment bounds for the solution $\bigl(X_m^N\bigr)_{m\ge 0}$ of~\eqref{eq:expoEuler}. It is worth mentioning that they are uniform with respect to $N\in\N$, $\tau\in(0,1)$ and $m\ge 0$.

\begin{propo}\label{propo:momentbounds-eem}
	Let Assumptions~\ref{ass:ergo} and~\ref{ass:Q} be satisfied. Let $p\in[1,\infty)$. For any $\beta\in[0,1]$, such that the condition~\eqref{eq:Qbeta} is satisfied, there exists $C_{\beta,p}\in(0,\infty)$ such that if the initial condition $x_0$ is a $\mathcal{F}_0$-measurable random variable, which satisfies $\E[\|x_0\|_{\HH^\beta}^{p}]<\infty$, then one has
\begin{equation}\label{eq:momentbounds-eem}
\underset{N\in\N}\sup~\underset{\tau\in(0,1)}\sup~\underset{m\ge 0}\sup~\E[\|X^N_m\|_{\HH^\beta}^{p}]\le C_{\beta,p}\Bigl(1+\E[\|x_0\|_{\HH^\beta}^{p}]\Bigr).
\end{equation}
\end{propo}

\begin{proof}
Like in the proof of Proposition~\ref{propo:momentbounds-exact}, let us introduce the auxiliary process $(Y^N_m)_{m\ge 0}$ given by
	\[
	Y^N_m=X^N_m-\IZ^N_m,\ \ m\ge 0,
	\]
where $(X^N_m)_{m\in\N}$ and $(Z^N_m)_{m\in\N}$ are given by \eqref{eq:expoEuler} and \eqref{eq:discreteconvolution}, respectively. One then has, for all $m\ge 0$,	
\[
Y^N_m=e^{t_m A_\gamma}X^N_0+\tau\sum_{\ell=0}^{m-1}e^{(t_m-t_\ell)\tau A_\gamma}F_N(Y^N_\ell+\IZ^N_\ell).
\]
Let us first treat the case $\beta=0$. Owing to the inequality~\eqref{eq:momentboundsdiscreteconvolution} from Lemma~\ref{lem:discreteconvolution} above, in order to prove the inequality~ \eqref{eq:momentbounds-eem} for $\beta=0$, it suffices to prove that one has
\[
\underset{N\in\N}\sup~\underset{\tau\in(0,1)}\sup~\underset{m\ge 0}\sup~\E[\|Y^N_m\|_{\HH}^{p}]\le C_{\beta,p}\Bigl(1+\E[\|x_0\|_{\HH}^{p}]\Bigr).
\]

Using the inequality~\eqref{eq:expodecrease} from Lemma~\ref{lem:semigroup} and the Lipschitz continuity property~\eqref{eq:LipF_HH0} of $F$ from $\HH$ to $\HH$, one obtains, for all $m\ge 0$,
\begin{align*}
\|Y^N_m\|_\HH\le \Mg e^{-\rg t_m}\|x_0\|_\HH+
\tau\Mg\sum_{\ell=0}^{m-1}e^{-\rg (t_m-t_\ell)}\bigl(\frac{L_f}{\sqrt{\lambda_1}}\|Y^N_\ell\|_\HH+\|F_N(\IZ^N_\ell)\|_\HH\bigr).
\end{align*}
Using Minkowskii's inequality and the moment bounds~\eqref{eq:momentboundsdiscreteconvolution} from Lemma~\ref{lem:discreteconvolution} for the discrete stochastic convolution $\IZ^N_m$, for any $p\in[0,\infty)$, there exists $C_p\in(0,\infty)$ such that one has for all $m\ge 0$
\begin{align*}
\bigl(\E[\|Y^N_m\|_{\HH}^{p}]\bigr)^{\frac1p}&\le C_p\Bigl(1+\bigl(\E[\|x_0\|_{\HH}^{p}]\bigr)^{\frac1p}\Bigr)+\tau\frac{\Mg L_f}{\sqrt{\lambda_1}}\sum_{\ell=0}^{m-1}e^{-\rg (m-\ell)\tau}\bigl(\E[\|Y^N_\ell\|_{\HH}^{p}]\bigr)^{\frac1p}.
\end{align*}
The condition~\eqref{eq:ergo} ensures that one has $\frac{\Lf \Mg}{\sqrt{\lambda_1}}<\rg$. As a consequence, applying the discrete Gronwall inequality to the mapping $m\mapsto e^{\rg m\tau}\bigl(\E[\|Y^N_m\|_{\HH}^{p}]\bigr)^{\frac1p}$, one obtains
\[
\underset{m\ge 0}\sup~\E[\|Y^N_m\|_{\HH}^{p}]\le C_{p}\Bigl(1+\E[\|x_0\|_{\HH}^{p}]\Bigr),
\]
where $C_p\in(0,\infty)$ is independent of the discretization parameters $\tau\in(0,1)$ and $N\in\N$.

Let us now treat the case $\beta\in(0,1]$, and assume that the condition~\eqref{eq:Qbeta} holds and the initial value $x_0$ satisfies $\E[\|x_0\|_{\HH^\beta}^{p}]<\infty$. Since $F$ is globally Lipschitz continuous from $\HH$ to $\HH^1$, see~\eqref{eq:LipF_HH1}, it is also globally Lipschitz continuous from $\HH$ to $\HH^\beta$, owing to the inequality~\eqref{eq:comparenorms}. Therefore there exists $C_\beta\in(0,\infty)$ such that, using the discrete mild formulation~\eqref{eq:mildexpoEuler}, for all $m\ge 0$ one has
\begin{align*}
\|X^N_m\|_{\HH^\beta}\le \Mg e^{-\rg m\tau}\|x_0\|_{\HH^\beta}+C_\beta
\tau\sum_{\ell=0}^{m-1}e^{-\rg(t_m-t_\ell)}\bigl(1+\|X^N_\ell\|_{\HH}\bigr)+\|\IZ_m^N\|_{\HH^\beta}.
\end{align*}
Using Minkowskii's inequality, the moment bounds for $\beta=0$ established above, and the moment bounds~\eqref{eq:momentboundsconvolution} from Lemma~\ref{lem:discreteconvolution} for the discrete stochastic convolution $\IZ_m^N$, one obtains
\begin{align*}
\underset{m\ge 0}\sup~\E[\|X^N_m\|_{\HH^\beta}^{p}]&\le C_{\beta,p}\Bigl(1+\E[\|x_0\|_{\HH^\beta}^{p}]\Bigr),
\end{align*}
where $C_{\beta,p}\in(0,\infty)$ is independent of the discretization parameters $\tau\in(0,1)$ and $N\in\N$.

The proof of Proposition~\ref{propo:momentbounds-eem} is thus completed.
\end{proof}

We now state an existence and uniqueness result of an invariant distribution denoted by $\mu_\infty^{N,\tau}$. Observe that the bounds in Proposition~\ref{propo:ergo-eem} below are uniform with respect to the discretization parameters $N\in\N$ and $\tau\in(0,1)$.

\begin{propo}\label{propo:ergo-eem}
Let Assumptions~\ref{ass:ergo} and~\ref{ass:Q} be satisfied. For any $N\in\N$ and for any $\tau\in(0,1)$, the $\HH_N$-valued Markov chain $\bigl(X^N_m\bigr)_{m\in\N}$ admits a unique invariant probability distribution $\mu_\infty^{N,\tau}$, which satisfies the following properties.
\begin{itemize}
\item If $x_0^N$ has distribution $\mu_\infty^{N,\tau}$, then for all $m\in\N$, $X^N_m$ has distribution $\mu_\infty^{N,\tau}$.
\item For all $\beta\in[0,1]$ such that the condition~\eqref{eq:Qbeta} is satisfied and all $p\in[1,\infty)$, one has
\begin{equation}
\underset{N\in\N}\sup~\underset{\tau\in(0,1)}\sup~\int \|x\|_{\HH^\beta}^{2p}d\mu_\infty^{N,\tau}(x)<\infty.
\end{equation}
\item There exists $C\in(0,\infty)$, such that if $\varphi:\HH\to\R$ is a Lipschitz continuous mapping, for all $N\in\N$, $\tau\in(0,1)$ and $m\ge 0$ and one has
\begin{equation}
\big|\E[\varphi(X^N_m)]-\int\varphi d\mu_\infty^{N,\tau}\big|\le C{\rm Lip}(\varphi)e^{-\rho(\gamma,f)m\tau}\bigl(1+\E[\|X^N(0)\|_{\HH}]\bigr),
\end{equation}
where $\rho(\gamma,f)>0$ is defined by~\eqref{eq:ergo2} and ${\rm Lip}(\varphi)=\underset{x_1,x_2\in \HH, x_1\neq x_2}\sup~\frac{|\varphi(x_2)-\varphi(x_1)|}{\|x_2-x_1\|_H}$.
\end{itemize}
\end{propo}

\subsection{Approximation of the invariant distribution}\label{sec:main}

In this section, we state the main results of this manuscript.

If $\varphi:\HH\to\R$ is a mapping of class $\mathcal{C}^2$, with bounded first and second order derivatives, define
\begin{equation}
\vvvert \varphi\vvvert_1=\underset{x\in\HH}\sup~\underset{h\in\HH}\sup~\frac{|D\varphi(x).h|}{\|h\|_\HH}~,\quad \vvvert \varphi\vvvert_2=\underset{x\in\HH}\sup~\underset{h,k\in\HH}\sup~\frac{|D^2\varphi(x).(h,k)|}{\|h\|_\HH \|k\|_\HH},
\end{equation}
where $D\varphi$ and $D^2\varphi$ are the first and second order derivatives of $\varphi$.

\begin{theo}\label{theo:galerkin}
Let Assumptions~\ref{ass:ergo},~\ref{ass:F} and~\ref{ass:Q} be satisfied.

Let $\beta\in[0,1]$ such that the condition~\eqref{eq:Qbeta} holds. There exists $C_\beta\in(0,\infty)$ such that if $\E[\|x_0\|_{\HH^{\beta}}]<\infty$, then for all $T\in(0,\infty)$, all $N\in\N$ and any mapping $\varphi:\HH\to\R$ of class $\mathcal{C}^2$, with bounded first and second order derivatives, one has
\begin{equation}\label{eq:theo_galerkin}
\big|\E[\varphi(X(T))]-\E[\varphi(X^N(T))]\big|\le C_\beta e^{-\rho(\gamma,f)T}\vvvert\varphi\vvvert_1\E[\|x_0\|_{\HH}]+C_\beta \bigl(\vvvert\varphi\vvvert_1+\vvvert\varphi\vvvert_2\bigr)\lambda_N^{-\beta\theta}\bigl(1+\E[\|x_0\|_{\HH^{\beta}}]\bigr),
\end{equation}
where $\theta=1/2$ in the general case and $\theta=1$ in the commutative noise case.

Moreover, for all $N\in\N$ and any mapping $\varphi:\HH\to\R$ of class $\mathcal{C}^2$, with bounded first and second order derivatives, one has
\begin{equation}\label{eq:theo_galerkin-measures}
\big|\int \varphi d\mu_\infty-\int \varphi d\mu_\infty^N \big|\le C_\beta\bigl(\vvvert\varphi\vvvert_1+\vvvert\varphi\vvvert_2\bigr)\lambda_N^{-\beta\theta}.
\end{equation}
\end{theo}

\begin{theo}\label{theo:eem}
Let Assumptions~\ref{ass:ergo},~\ref{ass:F} and~\ref{ass:Q} be satisfied. 

Let $\beta\in[0,1]$ such that the condition~\eqref{eq:Qbeta} holds. There exists $C_{\beta}\in(0,\infty)$ such that if $\E\bigl[\|x_0\|_{\HH^{\beta}}^2\bigr]<\infty$, then for all $\tau\in(0,1)$, $N\in\N$ and any mapping $\varphi:\HH\to\R$ of class $\mathcal{C}^2$, with bounded first and second order derivatives, for all $M\ge 0$ one has
\begin{equation}\label{eq:theo_eem}
\big|\E[\varphi(X^N(t_M))]-\E[\varphi(X^N_M)]\big|\le C_{\beta}\bigl(\vvvert\varphi\vvvert_1+\vvvert\varphi\vvvert_2\bigr)\tau^{\min(2\theta\beta,1)}\bigl(1+\E[\|x_0\|_{\HH^\beta}^2]\bigr),
\end{equation}
where $\theta=1/2$ in the general case and $\theta=1$ in the commutative noise case.

Moreover, for all $\tau\in(0,1)$, $N\in\N$ and any mapping $\varphi:\HH\to\R$ of class $\mathcal{C}^2$, with bounded first and second order derivatives, one has
\begin{equation}\label{eq:theo_eem-measures}
\big|\int \varphi d\mu_\infty^N-\int \varphi d\mu_\infty^{N,\tau}\big|\le C_{\beta}\bigl(\vvvert\varphi\vvvert_1+\vvvert\varphi\vvvert_2\bigr)\tau^{\min(2\theta\beta,1)}.
\end{equation}
\end{theo}

Combining Theorems~\ref{theo:galerkin} and~\ref{theo:eem}, one finally obtains the following result.
\begin{cor}\label{cor:full}
Let Assumptions~\ref{ass:ergo},~\ref{ass:F} and~\ref{ass:Q} be satisfied. Let $\beta\in[0,1]$ such that the condition~\eqref{eq:Qbeta} holds. There exists $C_{\beta,\gamma}\in(0,\infty)$ such that, for all $\tau\in(0,1)$, $N\in\N$ and any mapping $\varphi:\HH\to\R$ of class $\mathcal{C}^2$, with bounded first and second order derivatives, one has 
\begin{equation}\label{eq:cor_full-measures}
\big|\int \varphi d\mu_\infty-\int \varphi d\mu_\infty^{N,\tau}\big|\le C_{\beta}\bigl(\vvvert\varphi\vvvert_1+\vvvert\varphi\vvvert_2\bigr)\Bigl( \lambda_N^{-\beta\theta}+\tau^{\min(2\theta\beta,1)}\Bigr),
\end{equation}
where $\theta=1/2$ in the general case and $\theta=1$ in the commutative noise case.
\end{cor}

\section{Convergence analysis}\label{sec:proofs}

\subsection{Auxiliary results -- Kolmogorov equation}\label{sec:aux}

One of the key tools in the error analysis is to introduce solutions of Kolmogorov equations. Instead of dealing with infinite dimensional Kolmogorov equation, we consider the spectral Galerkin approximation introduced in Section~\ref{sec:galerkin} in order to justify all the computations in a finite dimensional setting. All the regularity estimates below do not depend on the spatial discretization parameter.

Let $\varphi:\HH\to\R$ be a mapping of class $\mathcal{C}^2$, with bounded first and second order derivatives. For all $N\in\N$, introduce the mapping $\varphi^N:\HH^N\to\R$ given by
\[
\varphi^N(x)=\varphi(x)~,\quad x\in\HH_N.
\]
For all $N\in\N$, $x\in\HH_N$ and $t\ge 0$, define
\begin{equation}\label{eq:Phi}
\Phi^N(t,x)=\E[\varphi^N(X_x^N(t))]
\end{equation}
where $\bigl(X_x^N(t)\bigr)_{t\ge 0}$ denotes the solution of~\eqref{eq:seegalerkin} with initial value $X_x^N(0)=x$. The mapping $\Phi^N$ is solution of the Kolmogorov equation
\begin{equation}\label{eq:Kolmogorov}
\begin{aligned}
\partial_t\Phi^N(t,x)&=\mathcal{L}^N\Phi^N(t,x)\\
&=D\Phi^N(t,x).\left(A_\gamma x+\IP_NF(x)\right)+\frac12\sum_{j\in\N}D^2\Phi^N(t,x).\left(\IP_N(0,Q^{1/2}e_j^Q),\IP_N(0,Q^{1/2}e_j^Q)\right)
\end{aligned}
\end{equation}
where $\mathcal{L}^N$ is the infinitesimal generator associated with the stochastic evolution equation~\eqref{eq:seegalerkin}, and $D\Phi^N(x)$ and $D^2\Phi^N(x)$ denote the first and second order derivatives of $\Phi^N$ with respect to the variable $x\in\HH_N$.

\begin{propo}\label{propo:Phi}
Let Assumptions~\ref{ass:ergo},~\ref{ass:F} and~\ref{ass:Q} be satisfied. There exists $C\in(0,\infty)$ such that for any mapping $\varphi:\HH\to\R$ of class $\mathcal{C}^2$ with bounded first and second order derivatives one has for all $t\ge 0$
\begin{align}
\underset{N\in\N}\sup~\underset{x\in\HH_N}\sup~\underset{h\in\HH_N}\sup~\frac{|D\Phi^N(t,x).h|}{\|h\|_\HH}&\le C\vvvert \varphi\vvvert_1 e^{-\rho(\gamma,f)t}\label{eq:DPhi}\\
\underset{N\in\N}\sup~\underset{x\in\HH_N}\sup~\underset{h,k\in\HH_N}\sup~\frac{|D^2\Phi^N(t,x).(h,k)|}{\|h\|_\HH\|k\|_\HH}&\le C\bigl(\vvvert \varphi\vvvert_1+\vvvert \varphi\vvvert_2\bigr) e^{-\rho(\gamma,f)t}\label{eq:D2Phi}.
\end{align}

\end{propo}

\begin{proof}
Let $N\in\N$. Owing to Assumption~\ref{ass:F}, the mapping $F_N:\HH_N\to\HH_N$ is twice continuously differentiable. For any $t\ge 0$, $x\in\HH_N$ and $h, k\in\HH_N$, one has
\begin{align*}
D\Phi^N(t,x).h&=\E[D\varphi(X_x^N(t)).\eta_h^N(t)],\\
D^2\Phi^N(t,x).(h,k)&=\E[D^2\varphi(X_x^N(t)).(\eta_h^N(t),\eta_k^N(t))]+\E[D\varphi(X_x^N(t)).\zeta_{h,k}^N(t)],
\end{align*}
where the $\HH_N$-valued processes $\bigl(\eta_h^N(t)\bigr)_{t\ge 0}$ and $\bigl(\zeta_{h,k}^N(t)\bigr)_{t\ge 0}$ are solutions of the first and second order variational equations
\begin{align*}
\frac{d\eta_h^N(t)}{dt}&=A_\gamma \eta_h^N(t)+DF_N(X_x^N(t)).\eta_h^N(t)\\
\frac{d\zeta_{h,k}^N(t)}{dt}&=A_\gamma \zeta_{h,k}^N(t)+DF_N(X_x^N(t)).\zeta_{h,k}^N(t)+D^2F_N(X_x^N(t)).(\eta_h^N(t),\eta_k^N(t)),
\end{align*}
with initial values $\eta_h^N(0)=h$ and $\zeta_{h,k}^N(0)=0$. Expressing the solution $\eta_h^N$ of the first variational equation in a mild formulation, for all $t\ge 0$ one has
\[
\eta_h^N(t)=e^{tA_\gamma}h+\int_0^t e^{(t-s)A_\gamma}DF_N(X_x^N(s)).\eta_h^N(s) ds
\]
and using the inequality~\eqref{eq:expodecrease} from Lemma~\ref{lem:semigroup}, the Lipschitz continuity property~\eqref{eq:LipF_HH0} of $F$ from $\HH$ to $\HH$, one obtains for all $t\ge 0$
\[
\|\eta_h^N(t)\|_{\HH}\le \Mg e^{-\rg t}\|h\|_{\HH}+\frac{\Mg\Lf}{\sqrt{\lambda_1}}\int_0^t e^{-\rg(t-s)}\|\eta_h^N(s)\|_\HH ds.
\]
Applying Gronwall's lemma to the mapping $t\mapsto e^{\rg t}\|\eta_h^N(t)\|_\HH$ and using the condition $\rho(\gamma,f)=\rg-\Mg\Lf\lambda_1^{-1/2}>0$ (see Equation~\eqref{eq:ergo2}) yields the following inequality: for all $t\ge 0$ one has
\begin{equation}\label{eq:eta}
\|\eta_h^N(t)\|_{\HH}\le \Mg e^{-\rho(\gamma,f)t}\|h\|_\HH.
\end{equation}
Combining the inequality~\eqref{eq:eta} with the expression of $D\Phi^N(t,x).h$ above yields
\[
\big|D\Phi^N(t,x).h\big|\le \vvvert\varphi\vvvert_1 \E[\|\eta_h^N(t)\|_\HH]\le \Mg\vvvert\varphi\vvvert_1 e^{-\rho(\gamma,f)t}\|h\|_\HH
\]
and concludes the proof of the first inequality~\eqref{eq:DPhi}.

The proof of the second inequality~\eqref{eq:D2Phi} is performed using similar arguments. On the one hand, using the inequality~\eqref{eq:eta} above, one has for all $t\ge 0$
\[
\big|\E[D^2\varphi(X_x^N(t)).(\eta_h^N(t),\eta_k^N(t))]\big|\le \vvvert\varphi\vvvert_2 \E[\|\eta_h^N(t)\|_\HH \|\eta_k^N(t)\|_\HH]\le \vvvert\varphi\vvvert_2\Mg^2 e^{-2\rho(\gamma,f)t}\|h\|_\HH\|k\|_\HH.
\]
On the other hand, expressing the solution $\zeta_{h,k}^N$ of the second variational equation in a mild formulation, for all $t\ge 0$ one has
\[
\zeta_{h,k}^N(t)=\int_0^t e^{(t-s)A_\gamma}DF_N(X_x^N(s)).\zeta_{h,k}^N(s) ds+\int_0^t e^{(t-s)A_\gamma}D^2F_N(X_x^N(s)).(\eta_h^N(s),\eta_k^N(s))ds
\]
and using the inequality~\eqref{eq:expodecrease} from Lemma~\ref{lem:semigroup}, the Lipschitz continuity property~\eqref{eq:LipF_HH0} of $F$ from $\HH$ to $\HH$ and the boundedness of the second-order derivative of $F$ from $\HH$ to $\HH$ (Assumption~\ref{ass:F}), and finally the inequality~\eqref{eq:eta} above, one obtains for all $t\ge 0$
\[
\|\zeta_{h,k}^N(t)\|_\HH\le \frac{\Mg\Lf}{\sqrt{\lambda_1}}\int_0^t e^{-\rg(t-s)}\|\zeta_{h,k}^N(s)\|_\HH ds +C\int_0^t e^{-\rg(t-s)}e^{-2\rho(\gamma,f) s}ds \|h\|_\HH \|k\|_\HH,
\]
where $C\in(0,\infty)$ is a positive real number (which does not depend on $N$).

Applying Gronwall's lemma to the mapping $t\mapsto e^{\rg t}\|\zeta_{h,k}^N(t)\|_\HH$ and using the condition $\rho(\lambda,f)=\rg-\Mg\Lf\lambda_1^{-1/2}>0$ (see Equation~\eqref{eq:ergo2}), one obtains the following inequality: for all $t\ge 0$ one has
\begin{equation}\label{eq:zeta}
\|\zeta_{h,k}^N(t)\|_{\HH}\le Ce^{-\rho(\gamma,f)t}\|h\|_\HH\|k\|_\HH.
\end{equation}
As a result, one has for all $t\ge 0$
\[
\big|\E[D\varphi(X_x^N(t)).\zeta_{h,k}^N(t)]\big|\le \vvvert\varphi\vvvert_1 \E[\|\zeta_h^N(t)\|_{\HH}]\le C\vvvert\varphi\vvvert_1  e^{-\rho(\gamma,f)t}\|h\|_\HH\|k\|_\HH,
\]
and combining the results finally concludes the proof of the second inequality~\eqref{eq:D2Phi}.
\end{proof}

\subsection{Proof of Theorem~\ref{theo:galerkin}}\label{sec:proofgalerkin}

Let $\varphi:\HH\to\R$ be a mapping of class $\mathcal{C}^2$, with bounded first and second order derivatives. Without loss of generality in the sequel it is assumed that $\vvvert\varphi\vvvert_1+\vvvert\varphi\vvvert_2\le 1$ to simplify the computations.

The objective of this section is to prove that the following claim holds: for all $\beta\in[0,1]$ such that the condition~\eqref{eq:Qbeta} holds, there exists $C_\beta\in(0,\infty)$ such that if $\E[\|x_0\|_{\HH^{\beta}}]<\infty$, then for all $T\in(0,\infty)$ and for all integers $\NN\ge N$ one has
\begin{equation}\label{eq:claim_galerkin}
\big|\E[\varphi(X^{\NN}(T))]-\E[\varphi(X^N(T))]\big|\le C_\beta e^{-\rho(\gamma,f)T}\E[\|x_0\|_{\HH}]+C_\beta \lambda_N^{-\beta\theta}\bigl(1+\E[\|x_0\|_{\HH^{\beta}}]\bigr),
\end{equation}
where $\theta=1/2$ in the general case and $\theta=1$ in the commutative noise case.

Observe that the right-hand side of~\eqref{eq:claim_galerkin} is independent of $\NN$, therefore using Lemma~\ref{lem:strongerror_galerkin} and letting $\NN\to\infty$ one obtains the inequality~\eqref{eq:theo_galerkin}.

Let us now prove~\eqref{eq:claim_galerkin}. The positive integer $\NN$ is introduced in order to use the results from Section~\ref{sec:aux} and in particular Proposition~\ref{propo:Phi}. Using the auxiliary function $\Phi^{\NN}$ given by~\eqref{eq:Phi} (with $N=\NN$), the left-hand side of~\eqref{eq:claim_galerkin} is written as
\begin{align*}
\E[\varphi(X^{\NN}(T))]&-\E[\varphi(X^N(T)]=\E[\Phi^{\NN}(T,X^{\NN}(0))]-\E[\Phi^{\NN}(0,X^N(T))]\\
&=\E[\Phi^{\NN}(T,X^{\NN}(0))]-\E[\Phi^{\NN}(T,X^N(0))]+\E[\Phi^{\NN}(T,X^N(0))]-\E[\Phi^{\NN}(0,X^N(T))].
\end{align*}
On the one hand, recall that $X^{\NN}(0)=\IP_{\NN}x_0$ and that $X^N(0)=\IP_Nx_0$, therefore using the inequality~\eqref{eq:DPhi} from Proposition~\ref{propo:Phi} one obtains for all $\NN\ge N$
\[
\big|\E[\Phi^{\NN}(T,X^{\NN}(0))]-\E[\Phi^{\NN}(T,X^N(0))]\big|\le Ce^{-\rho(\gamma,f)T}\E[\|X^{\NN}(0)-X^N(0)\|_\HH]\le Ce^{-\rho(\gamma,f)T}\E[\|x_0\|_\HH].
\]
On the other hand, using It\^o's formula, and the property that $\Phi^{\NN}$ is solution of the Kolmogorov equation~\eqref{eq:Kolmogorov} (with $N=\NN$), one has
\begin{align*}
\E[\Phi^{\NN}(0,X^N(T))]&-\E[\Phi^{\NN}(T,X^N(0))]=\int_{0}^{T}\E\left[D\Phi^{\NN}(T-t,X^N(t)).\Bigl(F_{\NN}(X^N(t))-F_N(X^N(t))\Bigr) \right] dt\\
&+\frac12\int_{0}^{T}\sum_{j\in\N}\E\left[D^2\Phi^{\NN}(T-t,X^N(t)).\Bigl(\IP_{\NN}(0,Q^{1/2}e_j^Q),\IP_{\NN}(0,Q^{1/2}e_j^Q)\Bigr) \right] dt\\
&-\frac12\int_{0}^{T}\sum_{j\in\N}\E\left[D^2\Phi^{\NN}(T-t,X^N(t)).\Bigl(\IP_N(0,Q^{1/2}e_j^Q),\IP_N(0,Q^{1/2}e_j^Q)\Bigr) \right] dt\\
&=\epsilon_{N,\NN}^1(T)+\epsilon_{N,\NN}^2(T)+\epsilon_{N,\NN}^3(T)
\end{align*}
where
\begin{align*}
\epsilon_{N,\NN}^1(T)&=\int_{0}^{T}\E\left[D\Phi^{\NN}(T-t,X^N(t)).\Bigl(F_{\NN}(X^N(t))-F_N(X^N(t))\Bigr) \right] dt\\
\epsilon_{N,\NN}^2(T)&=\frac12\int_{0}^{T}\sum_{j\in\N}\E\left[D^2\Phi^{\NN}(T-t,X^N(t)).\Bigl(\bigl(\IP_{\NN}-\IP_N\bigr)(0,Q^{1/2}e_j^Q),\IP_{\NN}(0,Q^{1/2}e_j^Q)\Bigr) \right] dt\\
\epsilon_{N,\NN}^3(T)&=\frac12\int_{0}^{T}\sum_{j\in\N}\E\left[D^2\Phi^{\NN}(T-t,X^N(t)).\Bigl(\IP_N(0,Q^{1/2}e_j^Q),\bigl(\IP_{\NN}-\IP_N\bigr)(0,Q^{1/2}e_j^Q)\Bigr) \right] dt.
\end{align*}
For the first error term, using the inequality~\eqref{eq:DPhi} from Proposition~\ref{propo:Phi} and the inequality~\eqref{eq:errorPN}, one has
\[
|\epsilon_{N,\NN}^1(T)|\le C_\beta\lambda_N^{-\beta}\int_{0}^{T}e^{-\rho(\gamma,f)(T-t)}\E[\|F(X^N(t))\|_{\HH^{2\beta}}]dt.
\]
If $\beta\in[0,1/2]$, using the inequality~\eqref{eq:Falphabis} and the moment bounds~\eqref{eq:momentbounds-galerkin} from Proposition~\ref{propo:momentbounds-galerkin}, one has
\[
\underset{N\in\N}\sup~\underset{t\ge 0}\sup~\E[\|F(X^N(t))\|_{\HH^{2\beta}}]\le C\Bigl(1+\underset{N\in\N}\sup~\underset{t\ge 0}\sup~\E[\|X^N(t)\|_{\HH}]\Bigr)\le C\bigl(1+\E[\|x_0\|_{\HH}]\bigr).
\]
If $\beta\in[1/2,1]$, using the inequality~\eqref{eq:Falpha} and the moment bounds~\eqref{eq:momentbounds-galerkin} from Proposition~\ref{propo:momentbounds-galerkin}, one has
\[
\underset{N\in\N}\sup~\underset{t\ge 0}\sup~\E[\|F(X^N(t))\|_{\HH^{2\beta}}]\le C_\beta\Bigl(1+\underset{N\in\N}\sup~\underset{t\ge 0}\sup~\E[\|X^N(t)\|_{\HH^\beta}]\Bigr)\le C\bigl(1+\E[\|x_0\|_{\HH^\beta}]\bigr).
\]
Therefore one obtains
\[
\underset{\NN\ge N}\sup~\underset{T\ge 0}\sup~|\epsilon_{N,\NN}^1(T)|\le C_\beta\lambda_N^{-\beta}\bigl(1+\E[\|x_0\|_{\HH^\beta}]\bigr).
\]

For the second error term, using the inequality~\eqref{eq:D2Phi} from Proposition~\ref{propo:Phi}, and the inequality~\eqref{eq:errorPN}, one obtains
\begin{align*}
|\epsilon_{N,\NN}^2(T)|&\le C\int_{0}^{T}e^{-\rho(\gamma,f)(T-t)}dt\sum_{j\in\N}\|(I-\IP_N\bigr)(0,Q^{1/2}e_j^Q)\|_\HH \|(0,Q^{1/2}e_j^Q)\|_{\HH}\\
&\le C_\beta\lambda_N^{-\beta/2}\sum_{j\in\N}\|(0,Q^{1/2}e_j^Q)\|_{\HH^{\beta}} \|(0,Q^{1/2}e_j^Q)\|_{\HH}\\
&\le C_\beta\lambda_N^{-\beta/2}\sum_{j\in\N}\|Q^{1/2}e_j^Q\|_{H^{\beta-1}}^2\\
&\le C_\beta\lambda_N^{-\beta/2}.
\end{align*}
In the commutative noise case ($\theta=1$), higher order of convergence is achieved: one has
\begin{align*}
|\epsilon_{N,\NN}^2(T)|&\le C\int_{0}^{T}e^{-\rho(\lambda,f)(T-t)}dt\sum_{j\in\N}\|(I-\IP_N\bigr)(0,Q^{1/2}e_j^Q)\|_\HH \|(0,Q^{1/2}e_j^Q)\|_{\HH}\\
&\le C_\beta\lambda_N^{-\beta}\sum_{j\in\N}\|(0,Q^{1/2}e_j^Q)\|_{\HH^{2\beta}} \|(0,Q^{1/2}e_j^Q)\|_{\HH}\\
&\le C_\beta\lambda_N^{-\beta}\sum_{j\in\N}\|Q^{1/2}e_j^Q\|_{H^{2\beta-1}}\|Q^{1/2}e_j^Q\|_{H^{-1}}\\
&\le C_\beta\lambda_N^{-\beta}\sum_{j\in\N}\|Q^{1/2}e_j^Q\|_{H^{\beta-1}}^2\\
&\le C_\beta\lambda_N^{-\beta}.
\end{align*}

The treatment of the third error term $\epsilon_{N,\NN}^3(T)$ is similar to the treatment of the second error term $\epsilon_{N,\NN}^2(T)$, the details are omitted. It is worth observing that $\epsilon_{N,\NN}^3(T)=0$ in the commutative noise case.

Gathering the estimates for the error terms obtained above, one obtains
\[
\big|\E[\Phi^{\NN}(0,X^N(T))]-\E[\Phi^{\NN}(T,X^N(0))]\big|\le C_\beta\lambda_N^{-\theta\beta}\bigl(1+\E[\|x_0\|_{\HH^\beta}]\bigr),
\]
where we recall that $\theta=1$ in the commutative noise case and $\theta=1/2$ otherwise. This concludes the proof of the weak error estimate~\eqref{eq:theo_galerkin}. Combining that result with Propositions~\ref{propo:ergo-exact} and~\ref{propo:ergo-galerkin}, choosing an arbitrary initial value $x_0$ and letting $T\to\infty$ yields~\eqref{eq:theo_galerkin-measures}. The proof of Theorem~\ref{theo:galerkin} is thus completed.

\subsection{Properties of an auxiliary process}

In order to prove Theorem~\ref{theo:eem}, it is convenient to introduce an auxiliary continuous time process $\bigl(\tilde{X}^N(t)\bigr)_{t\ge 0}$, defined as follows: for all $m\ge 0$ and all $t\in[t_m,t_{m+1}]$, set
\begin{equation}\label{eq:continuousversion-eem}
\tilde{X}^N(t)=e^{(t-t_m)A_\gamma}\bigl(X^N_m+(t-t_m)F_N(X^N_m)+\IP_N\bigl(\IW^Q(t)-\IW^Q(t_m)\bigr)\bigr). 
\end{equation}
Observe that for all $m\ge 0$ one has $\underset{t\to t_m}\lim~\tilde{X}^N(t)=X^N_m=\tilde{X}^N(t_m)$, therefore the auxiliary process $\bigl(\tilde{X}^N(t)\bigr)_{t\ge 0}$ is continuous. In addition, this process is a mild solution of the equation
\begin{equation}\label{eq:see-continuousversion-eem}
d\tilde{X}^N(t)=A_\gamma \tilde{X}^N(t)dt+e^{(t-t_{\ell(t)})A_\gamma}F_N(X^N_{\ell(t)})dt+e^{(t-t_{\ell(t)})A_\gamma}\IP_Nd\IW^Q(t),
\end{equation}
where we recall the notation $\ell(t)=\lfloor t/\tau\rfloor$.

For all $t\ge 0$, define $\tilde{u}^N(t)=\Pi_u \tilde{X}^N(t)$ and $\tilde{v}^N(t)=\Pi_v \tilde{X}^N(t)$.

Let us state and prove some properties of the auxiliary process.

\begin{propo}\label{propo:continuousversion-eem}
Let Assumptions~\ref{ass:ergo} and~\ref{ass:Q} be satisfied. Let $\beta\in[0,1]$, such that the condition~\eqref{eq:Qbeta} is satisfied. For all $p\in[1,\infty)$, there exists $C_{\beta,p}\in(0,\infty)$ such that if the initial value $x_0$ is a $\mathcal{F}_0$-measurable random variable, which satisfies $\E[\|x_0\|_{\HH^\beta}^2]<\infty$, then one has
\begin{equation}\label{eq:momentbounds-continuousversion-eem}
\underset{N\in\N}\sup~\underset{\tau\in(0,1)}\sup~\underset{t\ge 0}\sup~\E\bigl[\|\tilde{X}^N(t)\|_{\HH^{\beta}}^{p}\bigr]\le C_{\beta,p}\bigl(1+\E[\|x_0\|_{\HH^{\beta}}^p]\bigr).
\end{equation}
In addition, for all $N\in\N$, $\tau\in(0,1)$ and $t\ge 0$, one has
\begin{equation}\label{eq:incrementbounds-continuousversion-eem}
\bigl(\E[\|\tilde{X}^N(t)-X_{\ell(t)}^N\|_{\HH}^{p}]\bigr)^{\frac1p}\le C_{\beta,p}\tau^{\min(\beta,\frac12)}\bigl(1+\bigl(\E[\|x_0\|_{\HH^{\beta}}^p]\bigr)^{\frac1p}\bigr).
\end{equation}
Finally, for all $N\in\N$, $\tau\in(0,1)$ and $t\ge 0$, one has
\begin{equation}\label{eq:incrementbounds-u-continuousversion-eem}
\bigl(\E[\|\tilde{u}^N(t)-\tilde{u}^N(t_{\ell(t)})\|_{H^{\beta-1}}^{p}]\bigr)^{\frac1p}\le C_{\beta,p}\tau\bigl(1+\bigl(\E[\|x_0\|_{\HH^{\beta}}^p]\bigr)^{\frac1p}\bigr).
\end{equation}
\end{propo}

Combining the inequalities~\eqref{eq:incrementbounds-continuousversion-eem} and~\eqref{eq:incrementbounds-u-continuousversion-eem} when $\beta\in[0,\frac12]$, with the interpolation inequality~\eqref{eq:interp} and H\"older's inequality, one obtains the following result: for all $N\in\N$, $\tau\in(0,1)$ and $t\ge 0$, one has
\begin{equation}\label{eq:incrementbounds-u-continuousversion-eem2}
\bigl(\E[\|\tilde{u}^N(t)-\tilde{u}^N(t_{\ell(t)})\|_{H^{-\beta}}^{p}]\bigr)^{\frac1p}\le C_{\beta,p}\tau^{2\beta}\bigl(1+\bigl(\E[\|x_0\|_{\HH^{\beta}}^p]\bigr)^{\frac1p}\bigr).
\end{equation}
Indeed, if $\beta\in[0,\frac12]$, one has $-\beta\in[\beta-1,0]$.

The moment bounds~\eqref{eq:momentbounds-continuousversion-eem} for $\tilde{X}^N(t)$ are straightforward consequences of the moment bounds~\eqref{eq:momentbounds-eem} from Proposition~\ref{propo:momentbounds-eem}. The increment bounds~\eqref{eq:incrementbounds-continuousversion-eem} for $\tilde{X}^N(t)$ are then obtained using the inequality~\eqref{eq:regulsemigroup} from Lemma~\ref{lem:regulsemigroup}.

Observe that one has a term of size $\tau$ in the right-hand side of~\eqref{eq:incrementbounds-u-continuousversion-eem}. This plays an important role in the proof of Theorem~\ref{theo:eem} below. Note that one needs to study the increments of $t\mapsto \tilde{u}^N(t)$ in the $\|\cdot\|_{H^{\beta-1}}$ norm to obtain this result. A similar result for the exact mild solution $\bigl(X(t)\bigr)_{t\ge 0}$ of~\eqref{eq:see}: indeed owing to~\eqref{eq:spde} the increments of the mapping $t\mapsto u(t)=\Pi_uX(t)$ satisfy
\[
u(t_2)-u(t_1)=\int_{t_1}^{t_2}v(s)ds
\]
with $v(s)=\Pi_vX(s)$, and therefore
\begin{align*}
\bigl(\E[\|u(t_2)-u(t_1)\|_{H^{\beta-1}}^p]\bigr)^{\frac{1}{p}}&\le (t_2-t_1)\underset{s\ge 0}\sup~\bigl(\E[\|v(s)\|_{H^{\beta-1}}^p]\bigr)^{\frac{1}{p}}\le (t_2-t_1)\underset{s\ge 0}\sup~\bigl(\E[\|X(s)\|_{\HH^{\beta}}^p]\bigr)^{\frac{1}{p}}\\
&\le C_{\beta,p}(t_2-t_1)\bigl(1+\bigl(\E[\|x_0\|_{\HH^{\beta}}^p]\bigr)^{\frac1p}\bigr),
\end{align*}
owing to the moment bounds~\eqref{eq:momentbounds-exact} from Proposition~\ref{propo:momentbounds-exact}. The proof of the increment bounds~\eqref{eq:incrementbounds-u-continuousversion-eem} is based on a similar approach but requires to deal with some extra terms.

\begin{proof}
Using the definition~\eqref{eq:continuousversion-eem} of $\tilde{X}^N(t)$, the inequality~\eqref{eq:boundsemigroup}, the Lipschitz continuity property~\eqref{eq:LipF_HH1} of $F$ from $\HH$ to $\HH^1$, and the inequality~\eqref{eq:comparenorms}, one obtains the following inequality: for all $m\ge 0$ and all $t\in[t_m,t_{m+1}]$, one has
\begin{align*}
\|\tilde{X}^N(t)\|_{\HH^\beta}&\le \|X_m^N\|_{\HH^\beta}+\tau\|F(X_m^N)\|_{\HH^1}+\|\IW^Q(t)-\IW^Q(t_m)\|_{\HH^\beta}\\
&\le \|X_m^N\|_{\HH^\beta}+C\tau(1+\|X_m^N\|_{\HH})+\|\IW^Q(t)-\IW^Q(t_m)\|_{\HH^\beta}.
\end{align*}
Recall that the Wiener process $\bigl(\IW^Q(t)\bigr)_{t\ge 0}$ takes values in $\HH^\beta$, therefore the equality~\eqref{eq:incrementsWQ} holds. Then using the condition $t-t_m\le \tau$, the moment bounds~\eqref{eq:momentbounds-eem} from Proposition~\ref{propo:momentbounds-eem} and Minkowskii's inequality, one obtains the moment bounds~\eqref{eq:momentbounds-continuousversion-eem}.

Let us now prove the increment bounds~\eqref{eq:incrementbounds-continuousversion-eem} for $t\mapsto \tilde{X}^N(t)$. Using the inequality~\eqref{eq:regulsemigroup} from Lemma~\ref{lem:regulsemigroup} and the Lipschitz continuity property~\eqref{eq:LipF_HH0} of $F$ from $\HH$ to $\HH$, for all $m\ge 0$ and $t\in[t_m,t_{m+1}]$, one has
\begin{align*}
\|\tilde{X}^N(t)-X_m^N\|_{\HH}&\le \|(e^{(t-t_m)A_\gamma}-I)X_m^N\|_{\HH}+\tau\|F(X_m^N)\|_{\HH}+\|\IW^Q(t)-\IW^Q(t_m)\|_{\HH}\\
&\le C_\beta \tau^\beta\|X_m^N\|_{\HH^\beta}+C\tau(1+\|X_m^N\|_{\HH})+\|\IW^Q(t)-\IW^Q(t_m)\|_{\HH}.
\end{align*}
Using the equality~\eqref{eq:incrementsWQ}, the condition $t-t_m\le \tau$, the moment bounds~\eqref{eq:momentbounds-eem} from Proposition~\ref{propo:momentbounds-eem} and Minkowskii's inequality, one obtains the increment bounds~\eqref{eq:incrementbounds-continuousversion-eem}.

It remains to prove the increment bounds~\eqref{eq:incrementbounds-u-continuousversion-eem} for $t\mapsto \tilde{u}^N(t)$. The stochastic evolution equation~\eqref{eq:see-continuousversion-eem} for the auxiliary process $\bigl(\tilde{X}^N(t)\bigr)_{t\ge 0}$ is written as the following system: for all $m\ge 0$ and $t\in[t_m,t_{m+1}]$, one has
\begin{equation}\label{eq:continuousversion-eem1}
\left\lbrace
\begin{aligned}
&d\tilde{u}^N(t)=\tilde{v}^N(t)dt+\Pi_ue^{(t-t_m)A_\gamma}F_N(X^N_m)dt+\Pi_ue^{(t-t_m)A_\gamma}\IP_Nd\IW^Q(t),\\
&d\tilde{v}^N(t)=(-\IL \tilde{u}^N(t)-2\gamma \tilde{v}^N(t))\,dt+\Pi_ve^{(t-t_m)A_\gamma}F_N(X^N_m)dt+\Pi_ve^{(t-t_m)A_\gamma}\IP_Nd\IW^Q(t).
\end{aligned}
\right.
\end{equation}
As a consequence, since $u_m^N=\tilde{u}^N(t_m)$, one obtains the identities
\begin{align*}
\tilde{u}^N(t)-u^N_m&=\int_{t_m}^t \tilde{v}^N(s)ds+\int_{t_m}^t\Pi_ue^{(s-t_m)A_\gamma}F_N(X^N_m)ds+\int_{t_m}^t\Pi_ue^{(s-t_m)A_\gamma}\IP_Nd\IW^Q(s)\\
&=\int_{t_m}^t \tilde{v}^N(s)ds+\Pi_u\int_{t_m}^t\bigl(e^{(s-t_m)A_\gamma}-I\bigr)F_N(X^N_m)ds+\Pi_u\int_{t_m}^t\bigl(e^{(s-t_m)A_\gamma}-I\bigr)\IP_Nd\IW^Q(s),
\end{align*}
using the properties
\[
\Pi_uF_N(X^N_m)=0~,\quad \int_{t_m}^t\Pi_u\IP_Nd\IW^Q(s)=0
\]
to obtain the second equality. Note that $\beta-1\le 0$. Using the inequality~\eqref{eq:semigroup},~\eqref{eq:regulsemigroup} and~\eqref{eq:comparenorms}, the Lipschitz continuity property~\eqref{eq:LipF_HH0} of $F$ from $\HH$ to $\HH$, and Minkowskii's inequality, one obtains
\begin{align*}
\bigl(\E[\|\tilde{u}^N(t)-u_m^N\|_{H^{\beta-1}}^p]\bigr)^{\frac{1}{p}}&\le \int_{t_m}^{t}\bigl(\E[\|\tilde{v}^N(s)\|_{H^{\beta-1}}^p]\bigr)^{\frac{1}{p}}ds+C\tau\Bigl(1+\big(\E[\|X_m^N\|_{\HH}^p]\bigr)^{\frac{1}{p}}\Bigr)\\
&+\bigl(\E[\|\int_{t_m}^t\bigl(e^{(s-t_m)A_\gamma}-I\bigr)\IP_Nd\IW^Q(s)\|_{\HH^{\beta-1}}^p]\bigr)^\frac{1}{p}\\
&\le \tau\underset{s\ge 0}\sup~\E[\|\tilde{X}^N(s)\|_{\HH^\beta}^p]\bigr)^{\frac{1}{p}}+C\tau\Bigl(1+\big(\E[\|X_m^N\|_{\HH}^p]\bigr)^{\frac{1}{p}}\Bigr)\\
&+C\tau\bigl(\E[\|\IW^Q(t)-\IW^Q(t_m)\|_{\HH^\beta}^p]\bigr)^{\frac{1}{p}}.
\end{align*}
Using the moment bounds~\eqref{eq:momentbounds-continuousversion-eem} and~\eqref{eq:momentbounds-eem} for $\tilde{X}^N(s)$ and $X_m^N$ respectively, and the increment bounds~\eqref{eq:incrementsWQ} for $\IW^Q(t)-\IW^Q(t_m)$, one obtains~\eqref{eq:incrementbounds-u-continuousversion-eem}.

The proof of Proposition~\ref{propo:continuousversion-eem} is thus completed.
\end{proof}

\subsection{Proof of Theorem~\ref{theo:eem}}\label{sec:prooffully}

Let $\varphi:\HH\to\R$ be a mapping of class $\mathcal{C}^2$, with bounded first and second order derivatives. Without loss of generality in the sequel it is assumed that $\vvvert\varphi\vvvert_1+\vvvert\varphi\vvvert_2\le 1$ to simplify the computations. In this section, the value of the spatial discretization parameter $N\in\N$ is fixed.

The objective of this section is to prove the inequality~\eqref{eq:theo_eem} from Theorem~\ref{theo:eem}. Recall that the mapping $\Phi^N$ is defined by~\eqref{eq:Phi} (see Section~\ref{sec:aux}) and that the auxiliary process $\bigl(\tilde{X}^N(t)\bigr)_{t\ge 0}$ is defined by~\eqref{eq:continuousversion-eem}. Since $X^N(0)=X_N^0=\IP_Nx_0$, the left-hand side of~\eqref{eq:theo_eem} can be written as follows: for all $M\ge 0$ one has
\begin{equation}\label{eq:decomp}
\begin{aligned}
\E[\varphi(X^N(t_M))]-\E[\varphi(X^N_M)]&=\E[\Phi^N(t_M,X^N_0)]-\E[\Phi^N(0,X^N_M)]\\
&=\sum_{m=0}^{M-1}\Bigl(\E[\Phi^N(t_M-t_m,X_m^N)]-\E[\Phi^N(t_M-t_{m+1},X_{m+1}^N)]\Bigr)\\
&=\sum_{m=0}^{M-1}\Bigl(\E[\Phi^N(t_M-t_m,\tilde{X}^N(t_m))]-\E[\Phi^N(t_M-t_{m+1},\tilde{X}^N(t_{m+1})]\Bigr),
\end{aligned}
\end{equation}
owing to a standard telescoping sum argument.

Let $m\in\{0,\ldots,M-1\}$. Recall that the mapping $\Phi^N$ is solution of the Kolmogorov equation~\eqref{eq:Kolmogorov}. On the time interval $[t_m,t_{m+1}]$, $t\mapsto \tilde{X}^N(t)$ is solution of the stochastic evolution equation~\eqref{eq:see-continuousversion-eem} with $t_{\ell(t)}=t_m$. Applying It\^o's formula then yields the decomposition
\begin{equation}\label{eq:decompepsilon}
\E[\Phi^N(t_M-t_m,\tilde{X}^N(t_m))]-\E[\Phi^N(t_M-t_{m+1},\tilde{X}^N(t_{m+1}))]=\varepsilon_m^1+\varepsilon_m^2+\varepsilon_m^3+\varepsilon_m^4,
\end{equation}
with error terms defined by
\begin{align*}
\varepsilon_m^1&=\int_{t_m}^{t_{m+1}}\E\Bigl[D\Phi^N(t_M-t,\tilde{X}^N(t)).\Bigl(\bigl(e^{(t-t_m)A_\gamma}-I\bigr)F_N(X^N_m)\Bigr)\Bigr]dt\\
\varepsilon_m^2&=\int_{t_m}^{t_{m+1}}\E\bigl[D\Phi^N(t_M-t,\tilde{X}^N(t)).\bigl(F_N(X^N_m)-F_N(\tilde{X}^N(t))\bigr)\bigr]dt\\
\varepsilon_m^3&=\frac{1}{2}\int_{t_m}^{t_{m+1}}\E\Bigl[\sum_{n\in\N}D^2\Phi^N(t_M-t,\tilde{X}^N(t)).\bigl((e^{(t-t_m)A_\gamma}-I)\IP_N(0,Q^{1/2}e^Q_n),e^{(t-t_m)A_\gamma}\IP_N(0,Q^{1/2}e^Q_n)\bigr)\Bigr]dt\\
\varepsilon_m^4&=\frac{1}{2}\int_{t_m}^{t_{m+1}}\E\Bigl[\sum_{n\in\N}D^2\Phi^N(t_M-t,\tilde{X}^N(t)).\bigl(\IP_N(0,Q^{1/2}e^Q_n),(e^{(t-t_m)A_\gamma}-I)\IP_N(0,Q^{1/2}e^Q_n)\bigr)\Bigr]dt.
\end{align*}

Owing to the inequality~\eqref{eq:DPhi} from Proposition~\ref{propo:Phi}, and to the inequality~\eqref{eq:regulsemigroup} from Lemma~\ref{lem:regulsemigroup}, and using the Lipschitz continuity property~\eqref{eq:LipF_HH1} of $F$ from $\HH$ to $\HH^1$, for the error term $\epsilon_m^1$ one obtains the following upper bounds:
\begin{align*}
|\varepsilon_m^1|&\le \int_{t_m}^{t_{m+1}}\E\Bigl[\Big|D\Phi^N(t_M-t,\tilde{X}^N(t)).\Bigl(\bigl(e^{(t-t_m)A_\gamma}-I\bigr)F_N(X^N_m)\Bigr)\Big|\Bigr]dt\\
&\le \int_{t_m}^{t_{m+1}}e^{-\rho(\gamma,f)(t_M-t)}\E\bigl[\|\bigl(e^{(t-t_m)A_\gamma}-I\bigr)F_N(X^N_m)\|_{\HH}\bigr]dt\\
&\le C\tau \int_{t_m}^{t_{m+1}}e^{-\rho(\gamma,f)(t_M-t)}\E\bigl[\|F_N(X^N_m)\|_{\HH^1}\bigr]dt\\
&\le C\tau \int_{t_m}^{t_{m+1}}e^{-\rho(\gamma,f)(t_M-t)}\E\bigl[\bigl(1+\|X^N_m\|_{\HH}\bigr)\bigr]dt
\end{align*}
Finally, using the moment bounds~\eqref{eq:momentbounds-eem} from Proposition~\ref{propo:momentbounds-eem}, one has the following upper bounds for the error term $\varepsilon_m^1$:
\begin{equation}\label{eq:epsilon1}
|\varepsilon_m^1|\le C\tau\bigl(1+\E[\|x_0\|_\HH]\bigr) \int_{t_m}^{t_{m+1}}e^{-\rho(\gamma,f)(t_M-t)}dt~,\quad \forall m\in\{0,\ldots,M-1\}.
\end{equation}

To treat the error term $\varepsilon_m^2$, it is crucial to recall that $F(x)=(0,f(u))$ for all $x=(u,v)\in\HH=H\times H^{-1}$. As a consequence, using the inequality~\eqref{eq:DPhi} from Proposition~\ref{propo:Phi}, one obtains the following upper bounds:
\begin{align*}
|\varepsilon_m^2|&\le \int_{t_m}^{t_{m+1}}\E\bigl[\big|D\Phi^N(t_M-t,\tilde{X}^N(t)).\bigl(F_N(X^N_m)-F_N(\tilde{X}^N(t))\bigr) \big|\bigr]dt\\
&\le \int_{t_m}^{t_{m+1}}e^{-\rho(\gamma,f)(t_M-t)}\E\bigl[\|F_N(X^N_m)-F_N(\tilde{X}^N(t))\|_{\HH}\bigr]dt\\
&\le \int_{t_m}^{t_{m+1}}e^{-\rho(\gamma,f)(t_M-t)}\E\bigl[\|f(u^N_m)-f(\tilde{u}^N(t))\|_{\HH^{-1}}\bigr]dt.
\end{align*}
It is then necessary to treat separately the cases $\beta\in[0,\frac12]$ and $\beta\in[\frac12,1]$. On the one hand, if $\beta\in[0,\frac12]$, using the condition~\eqref{eq:Dfalpha} from Assumption~\ref{ass:F} (with $\alpha=\beta$), the inequality~\eqref{eq:incrementbounds-u-continuousversion-eem2} (which follows from the inequalities~\eqref{eq:incrementbounds-continuousversion-eem} and~\eqref{eq:incrementbounds-u-continuousversion-eem} from Proposition~\ref{propo:continuousversion-eem}), the moment bounds~\eqref{eq:momentbounds-eem} and~\eqref{eq:momentbounds-continuousversion-eem} and H\"older's inequality, one obtains for all $t\in[t_m,t_{m+1}]$
\begin{align*}
\E\bigl[\|f(u^N_m)-f(\tilde{u}^N(t))\|_{\HH^{-1}}\bigr]&\le C_\beta \bigl(1+\E[\|u^N_m\|_{\HH^{\beta}}^2]+\E[\|\tilde{u}(t)\|_{\HH^\beta}^2\bigr)^{\frac12}\bigl(\E[\|u^N_m-\tilde{u}^N(t)\|_{\HH^{-\beta}}^2]\bigr)^{\frac12}\\
&\le C_\beta\tau^{2\beta}\bigl(1+\E[\|x_0\|_{\HH^\beta}^2]\bigr).
\end{align*}
On the other hand, if $\beta\in[\frac12,1]$, using the inequality~\eqref{eq:Dfalpha2} (which follows from the condition~\eqref{eq:Dfalpha} from Assumption~\ref{ass:F}), the inequality~\eqref{eq:incrementbounds-u-continuousversion-eem} from Proposition~\ref{propo:continuousversion-eem}, the moment bounds~\eqref{eq:momentbounds-eem} and~\eqref{eq:momentbounds-continuousversion-eem} and H\"older's inequality, one obtains for all $t\in[t_m,t_{m+1}]$
\begin{align*}
\E\bigl[\|f(u^N_m)-f(\tilde{u}^N(t))\|_{\HH^{-1}}\bigr]&\le C_\beta \bigl(1+\E[\|u^N_m\|_{\HH^{\beta}}^2]+\E[\|\tilde{u}(t)\|_{\HH^\beta}^2\bigr)^{\frac12}\bigl(\E[\|u^N_m-\tilde{u}^N(t)\|_{\HH^{\beta-1}}^2]\bigr)^{\frac12}\\
&\le C_\beta\tau\bigl(1+\E[\|x_0\|_{\HH^\beta}^2]\bigr).
\end{align*}
Finally, one has the following upper bounds for the error term $\varepsilon_m^2$:
\begin{equation}\label{eq:epsilon2}
|\varepsilon_m^2|\le C_\beta\tau^{\min(2\beta,1)}\bigl(1+\E[\|x_0\|_{\HH^\beta}^2]\bigr) \int_{t_m}^{t_{m+1}}e^{-\rho(\gamma,f)(t_M-t)}dt~,\quad \forall m\in\{0,\ldots,M-1\}.
\end{equation}

The error terms $\epsilon_m^3$ and $\epsilon_m^4$ are treated with the same arguments. Using the inequality~\eqref{eq:D2Phi} from Proposition~\ref{propo:Phi} and the inequality~\eqref{eq:expodecrease} from Lemma~\ref{lem:semigroup}, one obtains
\[
|\epsilon_m^3|+|\epsilon_m^4|\le C\int_{t_m}^{t_{m+1}}e^{-\rho(\gamma,f)(t_M-t)}\sum_{n\in\N}\|(0,Q^{1/2}e_n^Q)\|_{\HH}\|(e^{(t-t_m)A_\gamma}-I)(0,Q^{1/2}e_n^Q)\|_{\HH}dt.
\]
Using the inequality~\eqref{eq:regulsemigroup} from Lemma~\ref{lem:regulsemigroup}, one then obtains
\begin{align*}
|\epsilon_m^3|+|\epsilon_m^4|&\le C_\beta\tau^\beta\int_{t_m}^{t_{m+1}}e^{-\rho(\gamma,f)(t_M-t)}dt\sum_{n\in\N}\|(0,Q^{1/2}e_n^Q)\|_{\HH}\|(0,Q^{1/2}e_n^Q)\|_{\HH^\beta}\\
&\le C_\beta\tau^\beta\int_{t_m}^{t_{m+1}}e^{-\rho(\gamma,f)(t_M-t)}dt\sum_{n\in\N}\|Q^{1/2}e_n^Q\|_{H^{-1}}\|Q^{1/2}e_n^Q\|_{H^{\beta-1}}\\
&\le C_\beta\tau^\beta\int_{t_m}^{t_{m+1}}e^{-\rho(\gamma,f)(t_M-t)}dt\sum_{n\in\N}\|Q^{1/2}e_n^Q\|_{H^{\beta-1}}^2\\
&\le C_\beta\tau^\beta\int_{t_m}^{t_{m+1}}e^{-\rho(\gamma,f)(t_M-t)}dt,
\end{align*}
assuming that $\beta$ satisfies the condition~\eqref{eq:Qbeta}.

Let us now get a higher rate of convergence in the commutative noise case. If $\beta\in[0,\frac12]$, one has
\begin{align*}
|\epsilon_m^3|+|\epsilon_m^4|&\le C_\beta\tau^{2\beta}\int_{t_m}^{t_{m+1}}e^{-\rho(\gamma,f)(t_M-t)}dt\sum_{n\in\N}\|(0,Q^{1/2}e_n^Q)\|_{\HH}\|(0,Q^{1/2}e_n^Q)\|_{\HH^{2\beta}}\\
&\le C_\beta\tau^{2\beta}\int_{t_m}^{t_{m+1}}e^{-\rho(\gamma,f)(t_M-t)}dt\sum_{n\in\N}\|Q^{1/2}e_n^Q\|_{H^{-1}}\|Q^{1/2}e_n^Q\|_{H^{2\beta-1}}\\
&\le C_\beta\tau^{2\beta}\int_{t_m}^{t_{m+1}}e^{-\rho(\gamma,f)(t_M-t)}dt\sum_{n\in\N}\|Q^{1/2}e_n^Q\|_{H^{\beta-1}}^2\\
&\le C_\beta\tau^{2\beta}\int_{t_m}^{t_{m+1}}e^{-\rho(\gamma,f)(t_M-t)}dt,
\end{align*}
while if $\beta\in[\frac12,1]$, one has
\begin{align*}
|\epsilon_m^3|+|\epsilon_m^4|&\le C\tau\int_{t_m}^{t_{m+1}}e^{-\rho(\gamma,f)(t_M-t)}dt\sum_{n\in\N}\|(0,Q^{1/2}e_n^Q)\|_{\HH}\|(0,Q^{1/2}e_n^Q)\|_{\HH^{1}}\\
&\le C\tau\int_{t_m}^{t_{m+1}}e^{-\rho(\gamma,f)(t_M-t)}dt\sum_{n\in\N}\|Q^{1/2}e_n^Q\|_{H^{-1}}\|Q^{1/2}e_n^Q\|_{H}\\
&\le C\tau\int_{t_m}^{t_{m+1}}e^{-\rho(\gamma,f)(t_M-t)}dt\sum_{n\in\N}\|Q^{1/2}e_n^Q\|_{H^{-\frac12}}^2\\
&\le C\tau\int_{t_m}^{t_{m+1}}e^{-\rho(\gamma,f)(t_M-t)}dt\sum_{n\in\N}\|Q^{1/2}e_n^Q\|_{H^{\beta-1}}^2\\
&\le C_\beta\tau\int_{t_m}^{t_{m+1}}e^{-\rho(\gamma,f)(t_M-t)}dt,
\end{align*}
using the inequality $-\frac12\le \beta-1$.

Recall that the definition of the parameter $\theta$: $\theta=\frac12$ in the general case and $\theta=1$ in the commutative noise case. Finally, one obtains the following error bounds for the error terms $\varepsilon_m^3$ and $\varepsilon_m^4$:
\begin{equation}\label{eq:epsilon34}
|\varepsilon_m^3|+|\varepsilon_m^4|\le C_\beta
\tau^{\min(2\theta\beta,1)}\int_{t_m}^{t_{m+1}}e^{-\rho(\gamma,f)(t_M-t)}dt~,\quad \forall m\in\{0,\ldots,M-1\}.
\end{equation}

It remains to gather the bounds~\eqref{eq:epsilon1},~\eqref{eq:epsilon2} and~\eqref{eq:epsilon34} for the error terms appearing in the right-hand side of the decomposition of the local weak error~\eqref{eq:decompepsilon}. Note that for all $M\ge 0$ one has
\[
\int_{0}^{t_M}e^{-\rho(\gamma,f)(t_M-t)}dt\le \int_{0}^{\infty}e^{-\rho(\gamma,f)t}dt=\rho(\gamma,f)^{-1}.
\]
Using the decomposition~\eqref{eq:decomp} of the global weak error, one obtains
\[
\big|\E[\varphi(X^N(t_M))]-\E[\varphi(X^N_M)]\big|\le C_\beta \tau^{\min(2\theta\beta,1)}\bigl(1+\E[\|x_0\|_{\HH^\beta}^2]\bigr),
\]
which concludes the proof of the error estimate~\eqref{eq:theo_eem}. Combining that result with Propositions~\ref{propo:ergo-eem} and~\ref{propo:ergo-galerkin}, choosing an arbitrary initial value $x_0$ and letting $M\to\infty$ yields~\eqref{eq:theo_eem-measures}. The proof of Theorem~\ref{theo:eem} is thus completed.

\begin{appendix}
\section{Proof of Lemma~\ref{lem:semigroup}}\label{sec:prooflemmasemigroup}

The techniques of the proof are similar to those used in~\cite[Lemma 5.1]{Salins:19}.

\begin{proof}
Let $\gamma\in(0,\infty)$. Owing to the expression~\eqref{eq:semigroup} of $e^{tA_\gamma}x$, it is sufficient to prove that for all $n\in\N$ the solution $t\mapsto (u_n(t),v_n(t))$ of~\eqref{eq:unvn} satisfies the following inequality: for all $n\in\N$ and all $t\ge 0$ one has
\begin{equation}\label{ineq}
|u_n(t)|^2+\lambda_n^{-1}|v_n(t)|^2\le (\Mg)^2 e^{-2\rg t}\bigl(|u_n(0)|^2+\lambda_n^{-1}|v_n(0)|^2\bigr).
\end{equation}
In the sequel, the value of $n\in\N$ is fixed. Introduce the auxiliary parameters
\[
\theta_{n,\gamma}=\frac{\min(\lambda_n,\gamma^2)}{2\gamma}~,\quad \Theta_{n,\gamma}=\theta_{n,\gamma}^2-2\gamma\theta_{n,\gamma}+\lambda_n,
\]
and the auxiliary functions, defined by
\[
\tilde{w}_{n,\gamma}(t)=e^{\theta_{n,\gamma}t}u_n(t),\quad \forall~t\ge 0.
\]
Straightforward computations show that $\tilde{w}_{n,\gamma}$ is solution of the linear second-order differential equation
\begin{equation}\label{eq:tildew}
\tilde{w}_{n,\gamma}''(t)+2\bigl(\gamma-\theta_{n,\gamma}\bigr)\tilde{w}_{n,\gamma}'(t)+\Theta_{n,\gamma}\tilde{w}_{n,\gamma}(t)=0,
\end{equation}
with initial values $\tilde{w}_{n,\gamma}(0)=u_n(0)$ and $\tilde{w}_{n,\gamma}'(0)=v_n(0)+\theta_{n,\gamma}u_n(0)$.

We claim that the following equalities hold: for all $t\ge 0$ one has
\begin{equation} \label{eq1}
|\tilde{w}_{n,\gamma}'(t)|^2+\Theta_{n,\gamma}|\tilde{w}_{n,\gamma}(t)|^2
+4\bigl(\gamma-\theta_{n,\gamma}\bigr)\int_0^t|\tilde{w}_{n,\gamma}'(s)|^2ds\le |\tilde{w}_{n,\gamma}'(0)|^2+\Theta_{n,\gamma}|\tilde{w}_{n,\gamma}(0)|^2
\end{equation}
and
\begin{equation}\label{eq2}
\begin{aligned}
|\tilde{w}_{n,\gamma}'(t)+2(\gamma-\theta_{n,\gamma})\tilde{w}_{n,\gamma}(t)|^2+\Theta_{n,\gamma}\Bigl(|\tilde{w}_{n,\gamma}(t)|^2&+4(\gamma-\Theta_{n,\gamma})\int_{0}^{t}|\tilde{w}_{n,\gamma}(s)|^2ds\Bigr)\\
&=|\tilde{w}_{n,\gamma}'(0)+2(\gamma-\theta_{n,\gamma})\tilde{w}_{n,\gamma}(0)|^2+\Theta_{n,\gamma}|\tilde{w}_{n,\gamma}(0)|^2.
\end{aligned}
\end{equation}
To prove~\eqref{eq1}, it suffices to multiply the left-hand side of~\eqref{eq:tildew} by $\tilde{w}_{n,\gamma}'(t)$ to obtain the identity
\[
\frac{1}{2}\frac{d|\tilde{w}_{n,\gamma}'(t)|^2}{dt}+2\bigl(\gamma-\theta_{n,\gamma}\bigr)|\tilde{w}_{n,\gamma}'(t)|^2+\frac{\Theta_{n,\gamma}}{2}\frac{d|\tilde{w}_{n,\gamma}(t)|^2}{dt}=0
\]
and to integrate from $0$ to $t$. To prove~\eqref{eq2}, it suffices to check the identity
\begin{align*}
\frac12\frac{d|\tilde{w}_{n,\gamma}'(t)+2(\gamma-\theta_{n,\gamma})\tilde{w}_{n,\gamma}(t)|^2}{dt}&=\Bigl(\tilde{w}_{n,\gamma}''(t)+2(\gamma-\theta_{n,\gamma})\tilde{w}_{n,\gamma}'(t)\Bigr)\Bigl(\tilde{w}_{n,\gamma}'(t)+2(\gamma-\theta_{n,\gamma})\tilde{w}_{n,\gamma}(t)\Bigr)\\
&=-\Theta_{n,\gamma}\tilde{w}_{n,\gamma}(t)\Bigl(\tilde{w}_{n,\gamma}'(t)+2(\gamma-\theta_{n,\gamma})\tilde{w}_{n,\gamma}(t)\Bigr)\\
&=-\Theta_{n,\gamma}\Bigl(\frac12\frac{d|\tilde{w}_{n,\gamma}(t)|^2}{dt}+2(\gamma-\theta_{n,\gamma})|\tilde{w}_{n,\gamma}(t)|^2\Bigr)
\end{align*}
and to integrate from $0$ to $t$.

We are now in position to proceed with the proof of the inequality~\eqref{ineq}. Two cases need to be treated separately.

First, assume that the condition $\gamma^2\le\lambda_n$ is satisfied. In that case one has
\begin{align*}
&\theta_{n,\gamma}=\gamma-\theta_{n,\gamma}=\frac{\gamma}{2}>0,~\\
&\Theta_{n,\gamma}=\theta_{n,\gamma}^2-2\gamma\theta_{n,\gamma}+\lambda_n=\lambda_n-\frac{3\gamma^2}{4}\ge \frac{\lambda_n}{4}>0.
\end{align*}
Owing to the lower bounds above and to the inequality~\eqref{eq1}, for all $t\ge 0$ one has
\begin{align*}
|\tilde{w}_{n,\gamma}'(t)|^2&\le |\tilde{w}_{n,\gamma}'(0)|^2+\bigl(\lambda_n-\frac{3\gamma^2}{4}\bigr)|\tilde{w}_{n,\gamma}(0)|^2\\
&\le C(\gamma)\Bigl(|v_n(0)|^2+\lambda_n|u_n(0)|^2\Bigr),
\end{align*}
where $C(\gamma)\in(0,\infty)$ is a positive real number (independent of $n\in\N$ and $t\ge 0$). In addition, owing to the lower bounds above and to the inequality~\eqref{eq2}, for all $t\ge 0$, one has
\begin{align*}
\frac{\lambda_n}{4}|\tilde{w}_{n,\gamma}(t)|^2&\le \Theta_{n,\gamma}|\tilde{w}_{n,\gamma}(t)|^2\\
&\le |\tilde{w}_{n,\gamma}'(0)+2(\gamma-\theta_{n,\gamma})\tilde{w}_{n,\gamma}(0)|^2+\Theta_{n,\gamma}|\tilde{w}_{n,\gamma}(0)|^2\\
&\le C(\gamma)\Bigl(|v_n(0)|^2+\lambda_n|u_n(0)|^2\Bigr).
\end{align*}
By the definition $\tilde{w}_{n,\gamma}(t)=e^{\theta_{n,\gamma}t}u_n(t)=e^{\gamma t/2}u_n(t)$ of the auxiliary function $\tilde{w}_{n,\gamma}$, one obtains the identities
\begin{align*}
&u_n(t)=e^{-\gamma t/2}\tilde{w}_{n,\gamma}(t)\\
&v_n(t)=u_n'(t)=e^{-\gamma t/2}\Bigl(\tilde{w}_{n,\gamma}'(t)-\frac{\gamma}{2}\tilde{w}_{n,\gamma}(t)\Bigr).
\end{align*}
Finally using the inequalities above yields the inequality~\eqref{ineq} when the condition $\gamma^2\le \lambda_n$ is satisfied.

Second, assume that the condition $\lambda_n\le \gamma^2$ is satisfied. In that case one has
\begin{align*}
&\theta_{n,\gamma}=\frac{\lambda_n}{2\gamma}\\
&\gamma-\theta_{n,\gamma}=\gamma(1-\frac{\lambda_n}{2\gamma^2})\ge \frac{\gamma}{2}\in[0,\frac{\gamma}{2}]\\
&\Theta_{n,\gamma}=\theta_{n,\gamma}^2-2\gamma\theta_{n,\gamma}+\lambda_n=\theta_{n,\gamma}^2=\frac{\lambda_n^2}{4\gamma^2}\in[0,\frac{\lambda_n}{4}].
\end{align*}
Owing to the lower bounds above and to the inequality~\eqref{eq1}, for all $t\ge 0$ one has
\begin{align*}
|\tilde{w}_{n,\gamma}'(t)|^2\le |\tilde{w}_{n,\gamma}'(0)|^2+\frac{\lambda_n^2|\tilde{w}_{n,\gamma}(0)|^2}{4\gamma^2}\le |\tilde{w}_{n,\gamma}'(0)|^2+\frac{\lambda_n|\tilde{w}_{n,\gamma}(0)|^2}{4}.
\end{align*}
In addition, using the bounds above and owing to the inequality~\eqref{eq2}, for all $t\ge 0$, one has
\begin{align*}
\gamma^2|\tilde{w}_{n,\gamma}(t)|^2&\le 4(\gamma-\theta_{n,\gamma})^2|\tilde{w}_{n,\gamma}(t)|^2\\
&\le 2|\tilde{w}_{n,\gamma}'(t)+2(\gamma-\theta_{n,\gamma})\tilde{w}_{n,\gamma}(t)|^2+2|\tilde{w}_{n,\gamma}'(t)|^2\\
&\le 2|\tilde{w}_{n,\gamma}'(0)+2(\gamma-\theta_{n,\gamma})\tilde{w}_{n,\gamma}(0)|^2+\frac{\lambda_n^2}{2\gamma^2}|\tilde{w}_{n,\gamma}(0)|^2+2|\tilde{w}_{n,\gamma}'(0)|^2+\frac{\lambda_n^2|\tilde{w}_{n,\gamma}(0)|^2}{2\gamma^2}\\
&\le C(\gamma)\Bigl(|\tilde{w}_{n,\gamma}'(0)|^2+|\tilde{w}_{n,\gamma}(0)|^2\Bigr),
\end{align*}
where $C(\gamma)\in(0,\infty)$ is a positive real number (independent of $n\in\N$ and $t\ge 0$).

By the definition $\tilde{w}_{n,\gamma}(t)=e^{\theta_{n,\gamma}t}u_n(t)=e^{\frac{\lambda_n t}{2\gamma}}u_n(t)$ of the auxiliary function $\tilde{w}_{n,\gamma}$, one obtains the identities
\begin{align*}
&u_n(t)=e^{-\frac{\lambda_n t}{2\gamma}}\tilde{w}_{n,\gamma}(t)\\
&v_n(t)=u_n'(t)=e^{-\frac{\lambda_n t}{2\gamma}}\Bigl(\tilde{w}_{n,\gamma}'(t)-\frac{\lambda_n}{2\gamma}\tilde{w}_{n,\gamma}(t)\Bigr),
\end{align*}
and using the inequalities $\lambda_1\le \lambda_n\le \gamma^2$ in the considered case, using the inequalities above yields the inequality~\eqref{ineq} when the condition $\lambda_n\le \gamma^2$ is satisfied.

Finally the inequality~\eqref{ineq} is proved in all situations (where the positive real numbers $\Mg$ and $\rg$ are independent of $n\in\N$ and $t\ge 0$). The proof of Lemma~\ref{lem:semigroup} is thus completed.
\end{proof}

\end{appendix}

\section{Acknowledgements}
The work of C.-E.~B. is partially supported by the projects ADA (ANR-19-CE40-0019-02) and SIMALIN (ANR-19-CE40-0016) operated by the French National Research Agency. The work of Z. Lei and S. Gan are supported by NSF of China (No. 11971488). The work of Z. Lei is supported by the China Scholarship Council (No. 202206370085).


\begin{thebibliography}{10}

\bibitem{AbdulleAlmuslimaniVilmart}
Assyr Abdulle, Ibrahim Almuslimani, and Gilles Vilmart.
\newblock Optimal explicit stabilized integrator of weak order 1 for stiff and
  ergodic stochastic differential equations.
\newblock {\em SIAM/ASA J. Uncertain. Quantif.}, 6(2):937--964, 2018.

\bibitem{AbdulleVilmartZygalakis}
Assyr Abdulle, Gilles Vilmart, and Konstantinos~C. Zygalakis.
\newblock High order numerical approximation of the invariant measure of
  ergodic {SDE}s.
\newblock {\em SIAM J. Numer. Anal.}, 52(4):1600--1622, 2014.

\bibitem{MR3484400}
Rikard Anton, David Cohen, Stig Larsson, and Xiaojie Wang.
\newblock Full discretization of semilinear stochastic wave equations driven by
  multiplicative noise.
\newblock {\em SIAM J. Numer. Anal.}, 54(2):1093--1119, 2016.

\bibitem{MR4454935}
S\'{e}bastien Boyaval, Sofiane Martel, and Julien Reygner.
\newblock Finite-volume approximation of the invariant measure of a viscous
  stochastic scalar conservation law.
\newblock {\em IMA J. Numer. Anal.}, 42(3):2710--2770, 2022.

\bibitem{C-E:14}
Charles-Edouard Br\'{e}hier.
\newblock Approximation of the invariant measure with an {E}uler scheme for
  stochastic {PDE}s driven by space-time white noise.
\newblock {\em Potential Anal.}, 40(1):1--40, 2014.

\bibitem{C-E:22}
Charles-Edouard Br\'{e}hier.
\newblock Approximation of the invariant distribution for a class of ergodic
  {SPDE}s using an explicit tamed exponential {E}uler scheme.
\newblock {\em ESAIM Math. Model. Numer. Anal.}, 56(1):151--175, 2022.

\bibitem{C-E:17}
Charles-Edouard Br\'{e}hier and Marie Kopec.
\newblock Approximation of the invariant law of {SPDE}s: error analysis using a
  {P}oisson equation for a full-discretization scheme.
\newblock {\em IMA J. Numer. Anal.}, 37(3):1375--1410, 2017.

\bibitem{C-E:16_2}
Charles-Edouard Br\'{e}hier and Gilles Vilmart.
\newblock High order integrator for sampling the invariant distribution of a
  class of parabolic stochastic {PDE}s with additive space-time noise.
\newblock {\em SIAM J. Sci. Comput.}, 38(4):A2283--A2306, 2016.

\bibitem{CCW}
Meng Cai, David Cohen, and Xiaojie Wang.
\newblock Strong convergence rates for a full discretization of stochastic wave
  equation with nonlinear damping.
\newblock {\em Preprint}, 2023.

\bibitem{Cerrai}
Sandra Cerrai.
\newblock {\em Second order {PDE}'s in finite and infinite dimension}, volume
  1762 of {\em Lecture Notes in Mathematics}.
\newblock Springer-Verlag, Berlin, 2001.
\newblock A probabilistic approach.

\bibitem{MR3605170}
Chuchu Chen, Jialin Hong, and Xu~Wang.
\newblock Approximation of invariant measure for damped stochastic nonlinear
  {S}chr\"{o}dinger equation via an ergodic numerical scheme.
\newblock {\em Potential Anal.}, 46(2):323--367, 2017.

\bibitem{Chen-Gan-Wang:20}
Ziheng Chen, Siqing Gan, and Xiaojie Wang.
\newblock A full-discrete exponential {E}uler approximation of the invariant
  measure for parabolic stochastic partial differential equations.
\newblock {\em Appl. Numer. Math.}, 157:135--158, 2020.

\bibitem{MR3033008}
David Cohen, Stig Larsson, and Magdalena Sigg.
\newblock A trigonometric method for the linear stochastic wave equation.
\newblock {\em SIAM J. Numer. Anal.}, 51(1):204--222, 2013.

\bibitem{MR3463447}
David Cohen and Llu\'{\i}s Quer-Sardanyons.
\newblock A fully discrete approximation of the one-dimensional stochastic wave
  equation.
\newblock {\em IMA J. Numer. Anal.}, 36(1):400--420, 2016.

\bibitem{CrisanDobsonOttobre:21}
D.~Crisan, P.~Dobson, and M.~Ottobre.
\newblock Uniform in time estimates for the weak error of the {E}uler method
  for {SDE}s and a pathwise approach to derivative estimates for diffusion
  semigroups.
\newblock {\em Trans. Amer. Math. Soc.}, 374(5):3289--3330, 2021.

\bibitem{Cui-Hong-Sun}
Jianbo Cui, Jialin Hong, and Liying Sun.
\newblock Weak convergence and invariant measure of a full discretization for
  parabolic {SPDE}s with non-globally {L}ipschitz coefficients.
\newblock {\em Stochastic Process. Appl.}, 134:55--93, 2021.

\bibitem{DPZergo}
G.~Da~Prato and J.~Zabczyk.
\newblock {\em Ergodicity for infinite-dimensional systems}, volume 229 of {\em
  London Mathematical Society Lecture Note Series}.
\newblock Cambridge University Press, Cambridge, 1996.

\bibitem{DPZ}
Giuseppe Da~Prato and Jerzy Zabczyk.
\newblock {\em Stochastic equations in infinite dimensions}, volume 152 of {\em
  Encyclopedia of Mathematics and its Applications}.
\newblock Cambridge University Press, Cambridge, second edition, 2014.

\bibitem{MR1500166}
Robert Dalang, Davar Khoshnevisan, Carl Mueller, David Nualart, and Yimin Xiao.
\newblock {\em A minicourse on stochastic partial differential equations},
  volume 1962 of {\em Lecture Notes in Mathematics}.
\newblock Springer-Verlag, Berlin, 2009.
\newblock Held at the University of Utah, Salt Lake City, UT, May 8--19, 2006,
  Edited by Khoshnevisan and Firas Rassoul-Agha.

\bibitem{feng2022higher}
Xiaobing Feng, Akash~Ashirbad Panda, and Andreas Prohl.
\newblock Higher order time discretization for the stochastic semilinear wave
  equation with multiplicative noise, 2022.

\bibitem{GHM}
Nathan Glatt-Holtz and Cecilia Mondaini.
\newblock Long-term accuracy of numerical approximations of spdes with the
  stochastic navier-stokes equations as a paradigm.
\newblock {\em Preprint}, 2023.

\bibitem{MR3884775}
Philipp Harms and Marvin~S. M\"{u}ller.
\newblock Weak convergence rates for stochastic evolution equations and
  applications to nonlinear stochastic wave, {HJMM}, stochastic
  {S}chr\"{o}dinger and linearized stochastic {K}orteweg--de {V}ries equations.
\newblock {\em Z. Angew. Math. Phys.}, 70(1):Paper No. 16, 28, 2019.

\bibitem{MR2671031}
Erika Hausenblas.
\newblock Weak approximation of the stochastic wave equation.
\newblock {\em J. Comput. Appl. Math.}, 235(1):33--58, 2010.

\bibitem{Hong-Wang}
Jialin Hong and Xu~Wang.
\newblock {\em Invariant measures for stochastic nonlinear {S}chr\"{o}dinger
  equations}, volume 2251 of {\em Lecture Notes in Mathematics}.
\newblock Springer, Singapore, 2019.
\newblock Numerical approximations and symplectic structures.

\bibitem{MR3608750}
Jialin Hong, Xu~Wang, and Liying Zhang.
\newblock Numerical analysis on ergodic limit of approximations for stochastic
  {NLS} equation via multi-symplectic scheme.
\newblock {\em SIAM J. Numer. Anal.}, 55(1):305--327, 2017.

\bibitem{MR4356894}
Ladislas Jacobe~de Naurois, Arnulf Jentzen, and Timo Welti.
\newblock Weak convergence rates for spatial spectral {G}alerkin approximations
  of semilinear stochastic wave equations with multiplicative noise.
\newblock {\em Appl. Math. Optim.}, 84(suppl. 2):S1187--S1217, 2021.

\bibitem{JentzenKloeden}
Arnulf Jentzen and Peter~E. Kloeden.
\newblock Overcoming the order barrier in the numerical approximation of
  stochastic partial differential equations with additive space-time noise.
\newblock {\em Proc. R. Soc. Lond. Ser. A Math. Phys. Eng. Sci.},
  465(2102):649--667, 2009.

\bibitem{MR2728063}
P.~E. Kloeden, G.~J. Lord, A.~Neuenkirch, and T.~Shardlow.
\newblock The exponential integrator scheme for stochastic partial differential
  equations: pathwise error bounds.
\newblock {\em J. Comput. Appl. Math.}, 235(5):1245--1260, 2011.

\bibitem{MR3434031}
Marie Kopec.
\newblock Weak backward error analysis for {L}angevin process.
\newblock {\em BIT}, 55(4):1057--1103, 2015.

\bibitem{MR2646102}
Mih\'{a}ly Kov\'{a}cs, Stig Larsson, and Fardin Saedpanah.
\newblock Finite element approximation of the linear stochastic wave equation
  with additive noise.
\newblock {\em SIAM J. Numer. Anal.}, 48(2):408--427, 2010.

\bibitem{LambertonPages}
Damien Lamberton and Gilles Pag\`es.
\newblock Recursive computation of the invariant distribution of a diffusion.
\newblock {\em Bernoulli}, 8(3):367--405, 2002.

\bibitem{LaurentVilmart}
Adrien Laurent and Gilles Vilmart.
\newblock Exotic aromatic {B}-series for the study of long time integrators for
  a class of ergodic {SDE}s.
\newblock {\em Math. Comp.}, 89(321):169--202, 2020.

\bibitem{LeimkuhlerMatthewsTretyakov}
B.~Leimkuhler, C.~Matthews, and M.~V. Tretyakov.
\newblock On the long-time integration of stochastic gradient systems.
\newblock {\em Proc. R. Soc. Lond. Ser. A Math. Phys. Eng. Sci.},
  470(2170):20140120, 16, 2014.

\bibitem{LPS}
Gabriel~J. Lord, Catherine~E. Powell, and Tony Shardlow.
\newblock {\em An introduction to computational stochastic {PDE}s}.
\newblock Cambridge Texts in Applied Mathematics. Cambridge University Press,
  New York, 2014.

\bibitem{MilsteinTretyakov}
G.~N. Milstein and M.~V. Tretyakov.
\newblock Computing ergodic limits for {L}angevin equations.
\newblock {\em Phys. D}, 229(1):81--95, 2007.

\bibitem{MR2224753}
Llu\'{\i}s Quer-Sardanyons and Marta Sanz-Sol\'{e}.
\newblock Space semi-discretisations for a stochastic wave equation.
\newblock {\em Potential Anal.}, 24(4):303--332, 2006.

\bibitem{Salins:19}
Michael Salins.
\newblock Smoluchowski--kramers approximation for the damped stochastic wave
  equation with multiplicative noise in any spatial dimension.
\newblock {\em Stochastics and Partial Differential Equations: Analysis and
  Computations}, 7(1):86--122, 2019.

\bibitem{Talay:02}
D.~Talay.
\newblock Stochastic {H}amiltonian systems: exponential convergence to the
  invariant measure, and discretization by the implicit {E}uler scheme.
\newblock volume~8, pages 163--198. 2002.
\newblock Inhomogeneous random systems (Cergy-Pontoise, 2001).

\bibitem{Talay}
Denis Talay.
\newblock Second-order discretization schemes of stochastic differential
  systems for the computation of the invariant law.
\newblock {\em Stochastics and Stochastic Reports}, 29(1):13--36, 1990.

\bibitem{Vilmart}
Gilles Vilmart.
\newblock Postprocessed integrators for the high order integration of ergodic
  {SDE}s.
\newblock {\em SIAM J. Sci. Comput.}, 37(1):A201--A220, 2015.

\bibitem{MR876085}
John~B. Walsh.
\newblock An introduction to stochastic partial differential equations.
\newblock In {\em \'{E}cole d'\'{e}t\'{e} de probabilit\'{e}s de
  {S}aint-{F}lour, {XIV}---1984}, volume 1180 of {\em Lecture Notes in Math.},
  pages 265--439. Springer, Berlin, 1986.

\bibitem{MR3353942}
Xiaojie Wang.
\newblock An exponential integrator scheme for time discretization of nonlinear
  stochastic wave equation.
\newblock {\em J. Sci. Comput.}, 64(1):234--263, 2015.

\bibitem{MR3276429}
Xiaojie Wang, Siqing Gan, and Jingtian Tang.
\newblock Higher order strong approximations of semilinear stochastic wave
  equation with additive space-time white noise.
\newblock {\em SIAM J. Sci. Comput.}, 36(6):A2611--A2632, 2014.

\end{thebibliography}

\end{document}